\documentclass[11pt]{article}
\usepackage{amsfonts,amssymb,amsmath,bm,amsthm}
\newtheorem{theorem}{Theorem}
\newtheorem{corollary}{Corollary}
\newtheorem{lemma}{Lemma}
\newtheorem{proposition}{Proposition}
\theoremstyle{remark}
\newtheorem{remark}{Remark}
\theoremstyle{definition}
\newtheorem{definition}{Definition}

\def\R{\mathbb{R}}
\def\C{\mathbb{C}}

\def\Q{\mathbb{Q}}
\def\Z{\mathbb{Z}}
\def\N{\mathbb{N}}
\def\I{\mathbb{I}}
\def\cA{\mathcal{A}}

\def\cH{\mathcal{H}}

\def\cM{\mathcal{M}}

\def\cC{\mathcal{C}}

\def\cL{\mathcal{L}}

\def\cR{\mathcal{R}}

\def\cF{\mathcal{F}}
\def\cV{\mathcal{V}}
\newcommand{\cFN}{\cF_n^{\bm N}}
\newcommand{\ve}{\varepsilon}
\newcommand{\vv}[1]{{\mathbf{#1}}}

\newcommand{\rank}{\operatorname{rank}}
\newcommand{\codim}{\operatorname{codim}}
\newcommand{\bmtheta}{{\bm\theta}}
\newcommand{\GL}{\mathrm{GL}}
\newcommand{\ddd}{m}
\newcommand{\BBG}{{\bf M}}
\newcommand{\Rp}{\R^+}    % positive real numbers
\newcommand{\ra}{R_{\alpha}}
\newcommand{\ba}{\beta_\alpha}
\newcommand{\al}{\alpha}
\renewcommand{\r}{\rho}

\usepackage{a4wide}

\begin{document}

\large

\title{\bf Systems of small linear forms and\\ Diophantine approximation on manifolds}

\author{V.~Beresnevich\footnote{Supported by EPSRC grant EP/J018260/1} \and V.~Bernik \and N.~Budarina}

\date{}%

\maketitle

\vspace*{-3ex}

\begin{abstract}
We develop the theory of Diophantine approximation for systems of simultaneously small linear forms, which coefficients are drawn from any given analytic non-degenerate manifolds. This setup originates from a problem of Sprind\v zuk from the 1970s on approximations to several real numbers by conjugate algebraic numbers. Our main result is a Khintchine type theorem, which convergence case is established without usual monotonicity constrains and the divergence case is proved for Hausdorff measures. The result encompasses several previous findings and, within the setup considered, gives the best possible improvement of a recent theorem of Aka, Breuillard, Rosenzweig and Saxc\'e on extremality.
\end{abstract}

\noindent{\small {\em Key words}: Diophantine approximation, Khintchine's theorem, simultaneously small linear forms, Hausdorff dimension, Mass Transference, regular systems and ubiquity}

\medskip

\noindent\emph{AMS Subject classification}: 11J83, 11J13, 11K60

\section{Introduction}

Diophantine approximation on manifolds dates back to a conjecture of Mahler \cite{mahler} from 1932 stating that for every $n\in\N$ and $\ve>0$ for almost every $x\in\R$ the inequality
\begin{equation}\label{vbA}
|P(x)|<H(P)^{-n-\ve}
\end{equation}
holds for finitely many polynomials $P=a_n x^n+\dots+a_1 x+a_0\in\Z[x]$ with $\deg P\le n$ only, where
$
H(P)=\max\{|a_i|:0\le i\le n\}
$
is the {\em height}\/ of $P$.
The conjecture was established by Sprind\v{z}uk in 1964, who also considered its $p$-adic and complex analogues \cite{spr1}.

More generally, given $\Psi:[0,+\infty)\to[0,+\infty)$, let $\mathcal{L}_{n}(\Psi)$ denote the set of $x\in\R$ such that the inequality
\begin{equation}\label{v4}
|P(x)|<\Psi(H(P))
\end{equation}
holds for infinitely many $P\in\Z[x]$, $\deg P\le n$.
Clearly, Sprind\v zuk's theorem simply means that $\mathcal{L}_{n}(h\mapsto h^{-n-\ve})$ is of Lebesgue measure zero for any $n\in\N$ and any $\ve>0$. Decades after Sprind\v zuk's proof, the following much more precise Khintchine type theorem was obtained as a results of \cite{beresnevich99}, \cite{beresnevich05} and \cite{bernik89}:

\bigskip

\noindent\textbf{Theorem A\,:} {\it
Let $n\in\N$, $\Psi:\N\to[0,+\infty)$ and $I\subset\R$ be any interval. Then
\begin{equation}\label{v3}
 \lambda_1(\mathcal{L}_{n}(\Psi)\cap I)=\left\{\begin{array}{cl}
 0 & \text{if ~$\sum_{h=1}^\infty h^{n-1}\Psi(h)<\infty$},\\[2ex]
 \lambda_1(I) & \text{if ~$\sum_{h=1}^\infty h^{n-1}\Psi(h)=\infty$ and $\Psi$ is monotonic}\,.
 \end{array} \right.
\end{equation}
}

\bigskip

\noindent Throughout $\lambda_m$ denotes Lebesgue measure over $\R^m$.
For $n=1$ \eqref{v3} is essentially Khintchine's classical result \cite{kh} on rational approximations to real numbers. Again, generalisations of \eqref{v3} were obtained for complex and $p$-adic variables, see \cite{bu-p-adic, bu3, bu6} and references within.

Clearly if a polynomial $P$ takes a small value at $x\in\R$, then one of the roots of $P$, say $\alpha$, must be close to $x$. There are various inequalities relating $|P(x)|$ and $|x-\alpha|$, see for instance \cite{spr1} and \cite{bernik89}. More generally, given a collection $x_1,\dots,x_m$ of real numbers, if the values $|P(x_j)|$ are simultaneously small, then every number $x_j$ from the collection is approximated by a root of $P$, say $\alpha_j$. In the case $P$ is irreducible over $\Q$, $\alpha_1,\dots,\alpha_m$ are conjugate. In this context a generalisation of Mahler's conjecture was established in \cite{bernik80} and reads as follows: {\it for any integers $n\ge m>1$ and any $\ve>0$ for almost all $(x_1,\dots,x_m)\in\R^m$ the inequality
\begin{equation}\label{vbB}
\max_{1\le j\le m}|P(x_j)|<H(P)^{-\frac{n+1-m}{m}-\ve}
\end{equation}
holds only for finitely many $P\in\Z[x]$ with $\deg P\le n$.}

\medskip

One of the goals of this paper is to obtain a complete analogue of Theorem~A for the setting of simultaneous approximations given by \eqref{vbB}. Although we shall consider the above problems in the much more general context of Diophantine approximation on manifolds, the result for polynomials is simpler, and we therefore present its full statement right away. To this end, define
\begin{equation}\label{sum}
S_{n,m}(\Psi)=\sum_{h=1}^{\infty}h^{n-m}\Psi^m(h)
\end{equation}
and let $\mathcal{L}_{n,m}(\Psi)$ be the set of $(x_1,\dots,x_m)\in \R^m$ such that
the inequality
 \begin{equation}\label{eq40+}
  \max_{1\le j\le m}|P(x_j)|<\Psi(H(P))
 \end{equation}
has infinitely many solutions $P\in\Z[x]$ with $\deg P\le n$.

\begin{theorem}\label{t1}
Let $n\ge m\ge 1$ be any integers, $\Psi:[0,+\infty)\to[0,+\infty)$ be any function and $B\subset\R^m$ be any ball. Then
\begin{equation}\label{v3+}
 \lambda_m\big(\mathcal{L}_{n,m}(\Psi)\cap B\big)=\left\{\begin{array}{cl}
 0 & \text{if ~$S_{n,m}(\Psi)<\infty$},\\[2ex]
 \lambda_m(B) & \text{if ~$S_{n,m}(\Psi)=\infty$ and $\Psi$ is monotonic}\,.
 \end{array} \right.
\end{equation}
\end{theorem}

\bigskip

\begin{remark}
The above theorem is not totally new. Indeed, the case of monotonic $\Psi$ was well investigated. In the case of $m=2$ the convergence case of the above result was previously obtained in \cite{borbat} under the assumption that $\Psi$ is monotonic. The analogue of Theorem~\ref{t1} for monotonic $\Psi$, for $m=3$ with $(x_1,x_2,x_3)\in\R\times\C\times\Q_p$ was obtained in \cite{bu4,bu5} and for systems with arbitrary number of real, complex and $p$-adic variables in \cite{bu+1,bu+2}. The main advances of this paper concern two main directions: establishing the convergence case for non-monotonic $\Psi$ and extending results to non-degenerate manifolds.
\end{remark}

\subsection{Generic systems of small linear forms}

In what follows $W(m,n;\Psi)$ will be the set of $m\times (n+1)$ real matrices $Y$ such that the system
\begin{equation}\label{v1+}
|Y_j\vv a|<\Psi(|\vv a|)\qquad(1\le j\le m)
\end{equation}
holds for infinitely many columns $\vv a=(a_0,\dots,a_n)^t\in\Z^{n+1}$,
where $Y_j$ denotes the $j$th row of $Y$ and $|\vv a|=\max_{0\le i\le n}|a_i|$.
Taking $Y_j=(1,x_j,\dots,x_j^n)$ transforms \eqref{v1+} into \eqref{eq40+}. Thus, \eqref{v1+} provides a natural framework for investigating the problems for polynomials discussed above.

Using a standard pigeonhole argument it is readily shown that for any $Y$ as above there is a $C>0$ such that $Y\in W(m,n;\Psi)$ for $\Psi(h)=C h^{-\frac{n+1-m}{m}}$, see \cite[Lemma~3]{D93}. The matrix $Y$ will be called {\em extremal}\/ if $Y\not\in W(m,n;\Psi_\tau)$ whenever $\Psi_\tau(h)=h^{-\tau}$ with $\tau>\frac{n+1-m}{m}$.
Using the Borel-Cantelli Lemma one can easily show that almost all $m\times (n+1)$ real matrices $Y$ are extremal.
The general theory for $W(m,n;\Psi)$ was initiated by Dickinson in \cite{D93} who found the Hausdorff dimension of this set. Her result was subsequently improved upon by Hussain and Levesley \cite{HL13}. A slightly simplified version of their main finding for the case $m\le n$ is now given.

\medskip

\noindent\textbf{Theorem B\,:} {\it Let $n\ge m\ge1$ be integers, $mn<s\le m(n+1)$ and $\Psi:\N\to[0,+\infty)$ be monotonic. Then
\begin{equation}\label{vb21}
\begin{array}[b]{l}
  \cH^s\big(W(m,n;\Psi)\cap\I^{m\times (n+1)}\big)=\\[1ex]
\hspace*{15ex}  =\left\{\begin{array}{cl}
  0&\text{if }\sum\limits_{r=1}^\infty \Psi(r)^{s-nm}\,r^{(n+1)m-s}<\infty\,,\\[2.5ex]
  \cH^s\big(\I^{m\times (n+1)}\big)&\text{if }\sum\limits_{r=1}^\infty \Psi(r)^{s-nm}\,r^{(n+1)m-s}=\infty\,.
         \end{array}
  \right.
\end{array}
\end{equation}
}

\medskip

\noindent Here $\I^{m\times(n+1)}$ is the set of $m\times(n+1)$ matrices $Y$ with entries restricted to $\I=[-\tfrac12,\tfrac12]$ and $\cH^s$ is the $s$-dimensional Hausdorff measure (see \S\ref{sec4} for further details). The case $s=m(n+1)$ of the above theorem corresponds to Lebesgue measure. Hence the following

\medskip

\noindent\textbf{Corollary\,C\,:} {\it Let $n,m$ and $\Psi$ be as in Theorem~B. Then
\begin{equation}\label{vb21+}
\begin{array}[b]{l}
  \lambda_{m(n+1)}\big(W(m,n;\Psi)\cap\I^{m\times (n+1)}\big)=\\[1ex]
\hspace*{15ex}
=\left\{\begin{array}{cl}
  0&\text{if }\sum\limits_{r=1}^\infty \Psi(r)^{m}<\infty\,,\\[2.5ex]
  1&\text{if }\sum\limits_{r=1}^\infty \Psi(r)^{m}=\infty\text{ and $\Psi$ is monotonic}\,.
         \end{array}
  \right.
\end{array}
\end{equation}
}

In should be noted that the monotonicity of $\Psi$ is not needed in the convergence case of Theorem~B and was later removed from the divergence case -- see \cite{HK13}.

\begin{remark}
The framework of Diophantine approximation given by \eqref{v1+} is different from the classical setting of the Khintchine-Groshev theorem, where each inequality is additionally reduced modulo $\Z$. In the latter case the torus geometry simplifies things a lot, see \cite{BBDV09, BV10,Dodson-survey}.
It is worth mentioning that the framework given by \eqref{v1+} has recently become of interest in applications in electronics, see for example \cite[Appendix~B]{OE15} and \cite[Appendix~B]{MHMK}. Also the theory for manifolds, that will shortly be discussed, plays an important role in backing some breakthrough discoveries on the degrees of freedom of Gaussian Interference Channels using real alignment -- see \cite{IA2,IA1}.
\end{remark}

The key goal of this paper is to develop the theory where every row $Y_j$ within \eqref{v1+} is restricted to a given analytic non-degenerate submanifold $\cM_j$. Our main result is a Khintchine type theorem, which convergence case is established without usual monotonicity constrains and the divergence case is proved for Hausdorff measures.

The theory for manifolds has been flourishing following the landmark work \cite{Kleinbock-Margulis-98:MR1652916} of Kleinbock and Margulis, who established the extremality of almost all rows/columns lying on any non-degenerate submanifold of $\R^n$. In particular, we have a Khintchine-Groshev type theory for rows/columns \cite{Ber02,BBKM,Bernik-Kleinbock-Margulis-01:MR1829381,bu1} and the theory of extremality for matrices \cite{ABRS1, ABRS2, BKM, KMW}. Relevant to the goals of this paper Aka, Breuillard, Rosenzweig and Saxc\'e \cite{ABRS2} establish that any analytic submanifold of $m\times(n+1)$ matrices is extremal is the sense defined just before Theorem~B above if and only if it is not contained in any so-called {\em constraining pencil}. The manifolds of matrices that we consider in this paper form a subclass of the manifolds considered in \cite{ABRS2}. Within this subclass our main result gives the best possible improvement of \cite{ABRS2}. Obtaining a Khintchine type result for more general submanifolds of matrices remains an interesting open problem for both convergence and divergence even for monotonic approximation functions $\Psi$.

\subsection{Main results}

Let $m\in\N$ and for $j=1,\dots,m$\, let
$$
\vv f_j=(f_{j,0},\dots,f_{j,n}):U_j\to\R^{n+1},\qquad\text{where $U_j\subset\R^{d_j}$ is an open ball.}
$$
Further, define
$$
U=U_1\times\dots\times U_m\subset\R^{d}\,,\qquad\text{where}\qquad d=d_1+\ldots+d_m.
$$
For each $(n+1)$-tuple $(a_0,\dots,a_n)\in\Z^{n+1}\setminus\{\vv0\}$, define the map
$$
F:U\to\R^m
$$
by setting
$$
F(\vv x_1,\dots,\vv x_m)=\left(\begin{array}{c}
F_1(\vv x_1)\\
\vdots\\
F_m(\vv x_m)
                               \end{array}
\right),
$$
where $F_j:U_j\to\R$ is given by
\begin{equation}\label{F_j}
F_j(\vv x_j)=\sum_{i=0}^na_if_{j,i}(\vv x_j)\,.
\end{equation}
Thus $F(\vv x_1,\dots,\vv x_m)$ is the product of the $m\times (n+1)$ matrix
$$
Y=Y(\vv x_1,\dots,\vv x_m):=\left(\begin{array}{ccc}
        f_{1,0}(\vv x_1) & \dots & f_{1,n}(\vv x_1) \\[0.5ex]
        \vdots &    & \vdots \\[0.5ex]
        f_{m,0}(\vv x_m) & \dots & f_{m,n}(\vv x_m)
      \end{array}
\right)
$$
and the column $\vv a=(a_0,\dots,a_n)^t$.
Throughout this paper $\cF=\cF(\vv f_1,\dots,\vv f_m)$ will denote the collection of all the maps $F$ as just defined with the coefficients $(a_0,\dots,a_n)$ ranging over all non-zero integer points. For a given $F\in\cF$ we will denote its defining integer coefficients by $a_0(F),\dots,a_n(F)$, or, when there is no risk of confusion, simply by $a_0,\dots,a_n$.
Finally, given $F\in\cF$, define the height of $F$ as
$$
H(F)=|\vv a|:=\max_{0\le j\le n} |a_j(F)|\,,
$$
The goal of this paper is to investigate the set $\cL(\cF,\Psi)$ consisting of point $(\vv x_1,\dots,\vv x_m)\in U$ such that
 \begin{equation}
  \label{eq40}
  \max_{1\le j\le m}|F_j(\vv x_j)|<\Psi(H(F))
 \end{equation}
holds for infinitely many $F\in\cF$, where $\Psi:\N\to\R^{+}$ is a given function.
The following theorem represents our main result.

\begin{theorem}\label{t2}
Let $n\ge m\ge1$ be integers.
Let $U=U_1\times\dots\times U_m$, $\vv f_1,\dots,\vv f_m$, $\cF=\cF(\vv f_1,\dots,\vv f_m)$, $\Psi$ and $\cL(\cF,\Psi)$ be as above. Suppose that for each $j=1,\dots,m$ the coordinate functions $f_{j,0},\dots,f_{j,n}$ of the map $\vv f_j$ are analytic and linearly independent over $\R$.  Then
$$
  \lambda_m(\cL(\cF,\Psi))=\left\{\begin{array}{cl}
                           0 & \text{if ~$S_{n,m}(\Psi)<\infty$}\,, \\[1ex]
                           \lambda_m(U) & \text{if ~$S_{n,m}(\Psi)=\infty$ and $\Psi$ is monotonic.}
                         \end{array}
  \right.
$$
\end{theorem}

\medskip

\begin{remark}
Note that taking $\vv f_j(x_j)=(1,x_j,x_j^2,\dots,x_j^n)$ for $j=1,\dots,m$ gives Theorem~\ref{t1}.
Furthermore, it is easy to see that both Theorem~A and Corollary~C and the divergence case of Theorem~B are the special cases of Theorem~\ref{t2}.
\end{remark}

\

The function $\Psi$ that governs the approximations in \eqref{eq40} can be fairly erratic even if it is monotonic. Before understanding the case of general $\Psi$ we shall look into the easier case when $\Psi(h)$ is of the form $h^{-v}$ for some positive parameter $v$. Apparently, this particular case holds the key to resolving the general case. We shall prove the following result which will get us half-way through to establishing the convergence case of Theorem~\ref{t2}, but is also of independent interest. Note that it allows different approximation `rates' in each inequality.

\begin{theorem}\label{thm-2}
Let $n\ge m\ge1$ be integers.
Let $U=U_1\times\dots\times U_m$, $\vv f_1,\dots,\vv f_m$ and $\cF=\cF(\vv f_1,\dots,\vv f_m)$ be as above. Suppose that for each $j=1,\dots,d$ the coordinate functions $f_{j,0},\dots,f_{j,n}$ of the map $\vv f_j$ are analytic and linearly independent over $\R$.  Further, let $v_1,\dots,v_m>0$ and $v_1',\dots,v_m'\ge-1$ be such that
\begin{equation}\label{ve}
v_1+\dots+v_m+v'_1+\dots+v'_m>n+1-2m.
\end{equation}
Then for any constants $c_1,\dots,c_m,c'_1,\dots,c'_m$ for almost every $(\vv x_1,\dots,\vv x_m)\in U$ the system of inequalities
\begin{equation}\label{vb0}
|F_j(\vv x_j)|<c_jH(F)^{-v_j},\qquad |F_j'(\vv x_j)|<c'_jH(F)^{-v_j'}\qquad (1\le j\le m)
\end{equation}
has only finitely many solutions $F\in\cF$.
\end{theorem}

\medskip

\begin{remark}
Using Minkowski's theorem for systems of linear forms or indeed the standard pigeonhole argument, one can readily show that if \eqref{ve} does not hold then there is a choice of positive constants $c_j,c'_j$ ($1\le j\le m$) such that \eqref{vb0} holds for infinitely many $F\in\cF$ on an open subset of $(\vv x_1,\dots,\vv x_m)$. Thus, \eqref{ve} is both sufficient and necessary assumption for the conclusion of Theorem~\ref{thm-2} to hold.
\end{remark}

\medskip

\begin{remark}
Using the inhomogeneous transference technique of \cite{BV10inh} one can generalise Theorem~\ref{thm-2} and indeed Theorem~\ref{t2} to inhomogeneous approximations. We leave the exploration of this research avenue to an interested reader, but see \cite{BBV13} and \cite{BKM} for a related content.
\end{remark}

\medskip

\begin{remark}
If some of $v_j'$ are equal to $-1$ the statement of Theorem~\ref{thm-2} will hold if the corresponding inequalities $|F_j'(\vv x_j)|<c'_jH(F)^{-v_j'}$ within
\eqref{vb0} are omitted. This is due to the fact $|F_j'(\vv x_j)|\ll H(F_j)$ anyway, where the implied constant in the Vinogradov symbol $\ll$ can be made absolute on any compact subset of $U_j$.
\end{remark}

\begin{remark}\label{rem3}
The condition of linear independence that we impose on every $(n+1)$-tuple $(f_{j,0},\dots,f_{j,n})$ of analytic functions is often referred to as {\em non-degeneracy} \cite{Kleinbock-Margulis-98:MR1652916}. It is well known that for every $j$ the domain $U_j$ of the non-degenerate analytic map $\vv f_j$ can be foliated by a continuous family of polynomial curves so that the restriction of $\vv f_j$ onto any of these curves is still an analytic non-degenerate map. Then, using such a foliation together with Fubini's theorem reduces the general case of Theorems~\ref{t2} and~\ref{thm-2} to the case of curves, that is the case when every map $\vv f_j$ is of a single real variable, say $x_j$. Specifically, this can be done in a fairly straightforward manner by making use of the Fibering Lemma of \cite[page~1206]{BerInvent} (see also Lemma~\ref{FL} below). The upshot of this remark is that in the course of establishing Theorems~\ref{t2} we can assume without loss of generality that every $U_j$ is an interval in $\R$, that is $d_1=\dots=d_m=1$.
\end{remark}

\section{Proof of Theorem~\ref{thm-2}}

\subsection{Preliminaries}\label{iprp}

Consider the following system of inequalities
\begin{equation}\label{e:082}
    \Big|\sum_{j=1}^k g_{i,j}(\vv x)\,a_j\Big|\le\theta_i\quad
    (1\le i\le k)\,,
\end{equation}
where $g_{i,j}:U\to\R$ are functions of $\vv x=(x_1,\dots,x_m)$ defined on an
open subset $U$ of $\R^{\ddd}$, $a_1,\dots,a_k$ are real variables and
$\bmtheta=(\theta_1,\dots,\theta_k)$ is a fixed $k$-tuple of
positive numbers. We will assume that $G(\vv x):=(g_{i,j}(\vv x))_{1\le i,j\le k}\in \GL_{k}(\R)$ for every
$\vv x\in U$.
Let
$$
\cA(G,\bmtheta):=\{\vv x\in U:\exists\ \vv
a=(a_1,\dots,a_k)\in\Z^{k}\smallsetminus\{\vv0\}\text{ satisfying
(\ref{e:082})}\}.
$$
We will be interested in estimating $\lambda_m(B\cap\cA(G,\bmtheta))$ in terms of $\lambda_m(B)$. For this purpose, we will use a general answer to this problem provided in \cite{BerAnn}, which in turn is a consequence of the even more general theorem of Kleinbock and Margulis from \cite{Kleinbock-Margulis-98:MR1652916}.
Following \cite[\S5]{BerAnn}, let
\begin{equation}\label{e:083}
\theta=(\theta_1\cdots\theta_k)^{\frac1k}.
\end{equation}
Given $\vv x\in U$ and a subspace
$V$ of $\R^k$ with $\codim V=r$, where $1\le r<k$, define
\begin{equation}\label{e:084}
 \Theta_{\bmtheta}(\vv x,V):=\min\left\{\theta^{-r}\prod_{i=1}^r\theta_{j_i}\,:\,
 \begin{array}{l} \{j_1,\dots,j_r\}\subset\{1,\dots,k\} \ \ \text{such
that}\\[0.5ex]
 V\oplus\cV\big(\vv g_{j_1}(\vv x),\dots,\vv g_{j_r}(\vv x)\big)=\R^k\,,
\end{array}\right\}
\end{equation}
where $\cV\big(\vv g_{j_1},\dots,\vv g_{j_r}\big)$ is the
subspace of $\R^k$ spanned by $\vv g_{j_1},\dots,\vv g_{j_r}$. Given $\vv x_0\in U$, let
\begin{equation}\label{e:085}
\widehat\Theta_{\bmtheta}(\vv x_0,V):=\liminf_{\vv x\to\vv
x_0}\Theta_{\bmtheta}(\vv x,V)\qquad\text{and}\qquad
\widehat\Theta_{\bmtheta}(\vv
x_0):=\sup_V\widehat\Theta_{\bmtheta}(\vv x_0,V),
\end{equation}
where the supremum is taken over subspaces $V\subset
\R^k$ with $1\le\codim V<k$.

\medskip

The following general result appears as Theorem~5.2 in \cite{BerAnn}.

\begin{proposition}[Theorem~5.2 in \cite{BerAnn}]\label{p:1}
Let $U$ be an open subset of\/ $\R^{\ddd}$, $G:U\to\GL_{k}(\R)$ be
an analytic map and $\vv x_0\in U$. Then there is a ball $B_0\subset
U$ centred at $\vv x_0$ and constants $K_0,\alpha>0$ such that for
any ball $B\subset B_0$ there is $\delta=\delta(B,G)>0$
such that for any $k$-tuple $\bmtheta=(\theta_1,\dots,\theta_k)$ of
positive numbers
\begin{equation}\label{e:086}
\lambda_{\ddd}\Big(B\cap\cA(G,\bmtheta)\Big)\le K_0\,\Big(1+\big(\sup_{\vv
x\in B}\widehat\Theta_{\bmtheta}(\vv
x)/\delta\big)^\alpha\Big)\,\theta^{\alpha}\,\lambda_{\ddd}(B)\,.
\end{equation}
\end{proposition}

When applying the above theorem, estimating $\widehat\Theta_{\bmtheta}(\vv x)$ becomes the main task. For example, it was shown in \cite[Lemmas~5.6 and 5.7]{BerAnn} that if $g_1,\dots,g_k$ is a collection of real analytic linearly independent over $\R$ functions of one variable and $G(x)=(g^{(i-1)}_j(x))_{1\le i,j\le k}$, then
\begin{equation}\label{e:107}
    \widehat\Theta_{\bmtheta}(x_0)\ \le\ \tilde\Theta:=\max_{1\le
    r\le k-1}\ \frac{\theta_1\cdots\theta_{r}}{\theta^r}
\end{equation}
for every $x_0$. It was subsequently shown in \cite[Lemma~2]{BBG10} that the parameter $\tilde\Theta$ can be further estimated by a simple expression as follows
\begin{equation}\label{omega}
    \tilde\Theta\le \max\left\{\frac{\theta_1}{\theta^k},\frac{1}{\theta_k}\right\}
\end{equation}
provided that $\theta$, that is given by \eqref{e:083}, is less than or equal to $1$ and the $k$-tuple $\bmtheta=(\theta_1,\dots,\theta_k)$ satisfies the following

\medskip

\noindent\textbf{Property M: } Given a $k'$-tuple $(\theta_1',\dots,\theta_{k'}')$ of positive real numbers with $k'\ge1$ we will say that it \emph{satisfies property \BBG} if there exists an integer $\ell'$ with $0\le \ell'\le k'$ such that
\begin{equation}\label{bbg}
\theta'_1,\dots,\theta'_{\ell'}\le1 \quad\text{while}\quad\theta'_{\ell'+1},\dots,\theta'_{k'}\ge1\,.
\end{equation}

To end this discussion we now formally state a lemma which is formally established within the proof of Lemma~5.7 in \cite{BerAnn}.

\begin{lemma}\label{l3}
Let $\vv g=(g_1,\dots,g_k)$ be a $k$-tuple of real analytic linearly independent over $\R$ functions defined on an interval $I$. Let $V$ be a linear subspace of\/ $\R^k$ with $\codim V=r\le k-1$. Define
$$
G(V)=\{x\in I:\dim V\oplus\cV(\vv g^{(0)}(x),\dots,\vv g^{(r-1)}(x))=k\}\,.
$$
Then $I\setminus G(V)$ is at most countable.
\end{lemma}

To some extent this lemma was the key to showing \eqref{e:107} and its following generalisation will be used in our subsequent arguments.

\begin{lemma}\label{l3m}
Let $\vv g=(g_1,\dots,g_k)$ be a $k$-tuple of real analytic linearly independent over $\R$ functions defined on an interval $I$. Let $V$ be a linear subspace of $\R^k$, $0\le s\le \codim V$ be an integer and
$$
G(V,s)=\{x\in I:\dim V\oplus\cV(\vv g^{(0)}(x),\dots,\vv g^{(s-1)}(x))=\dim V+s\}\,.
$$
Then $I\setminus G(V,s)$ is at most countable.
\end{lemma}

\begin{proof}
There is nothing to prove if $\dim V=k$ as in this case $s=0$ and $G(V,s)=I$. Assume that $\dim V<k$. Take any subspace $W$ of $\R^k$ such that $W\supset V$ and $\dim W+s=k$. By Lemma~\ref{l3}, $I\setminus G(W,s)$ is at most countable. Note that, since $V\subset W$, whenever $W\oplus\cV(\vv g^{(0)}(x),\dots,\vv g^{(s-1)}(x))=k=\dim W+s$, we necessarily have that $V\oplus\cV(\vv g^{(0)}(x),\dots,\vv g^{(s-1)}(x))=\dim V+s$. Hence $G(W,s)\subset G(V,s)$ and consequently $I\setminus G(V,s)\subset I\setminus G(W,s)$ is at most countable.
\end{proof}

\subsection{An application of Proposition~\ref{p:1}}

From now on let $k=n+1$, $m\le n$ and for $j=1,\dots,m$\, let $U_j$ be an interval in $\R$ and let $\vv f_j=(f_{j,0},\dots,f_{j,n}):U_j\to\R$ be an $(n+1)$-tuple of real analytic functions linearly independent over $\R$. Choose any integers $\ell_i\ge0$ such that
\begin{equation}\label{vb1}
\sum_{j=1}^m(\ell_j+1)=n+1~(=k)\,.
\end{equation}
Define $G(\vv x)$ of $\vv x=(x_1,\dots,x_m)\in U=U_1\times\dots\times U_m$ by setting
\begin{equation}\label{G}
G(x_1,\dots,x_m)=\left(\begin{array}{l}
                        \left.\begin{array}{c}
                          \vv f_1(x_1) \\
                          \vdots\\
                          \vv f_1^{(\ell_1)}(x_1)
                        \end{array}\right\}\ \ell_1+1\text{ times}\\[5ex]
                        \qquad\qquad\quad\cdot\\[0ex]
                        \qquad\qquad\quad\cdot\\[0ex]
                        \qquad\qquad\quad\cdot\\[2ex]
                        \left.\begin{array}{c}
                          \vv f_m(x_m) \\
                          \vdots\\
                          \vv f_m^{(\ell_m)}(x_m)
                        \end{array}\right\}\ \ell_m+1\text{ times}
                       \end{array}
\right)\,.
\end{equation}
Thus, the first $\ell_1+1$ rows are the map $\vv f_1(x_1)$ and its derivatives $\vv f_1^{(i)}(x_1)$ up to the order $\ell_1$; then we have $\ell_2+1$ rows which are the map $\vv f_2(x_2)$ and its derivatives $\vv f_2^{(i)}(x_2)$ up to the order $\ell_2$; and so on.
By \eqref{vb1}, $G(\vv x)$ is an $(n+1)\times(n+1)$ (that is $k\times k$) matrix.

Further, we choose any $k$-tuple $\bmtheta$ of positive real numbers which components will be combined into $m$ groups as follows
\begin{equation}\label{theta}
\bmtheta=(\tilde\theta_1,\dots,\tilde\theta_m)=(\underbrace{\theta_{1,0},\dots, \theta_{1,\ell_1}}_{\textstyle\tilde\theta_1},\ .\ . \ .\ ,\underbrace{\theta_{m,0},\dots,\theta_{m,\ell_m}}_{\textstyle\tilde\theta_m})\,.
\end{equation}
Thus
$$
\tilde\theta_j=(\theta_{j,0},\dots,\theta_{j,\ell_j})\quad\text{for each }j=1,\dots,m.
$$
In this case the parameter $\theta$ introduced by \eqref{e:083} is as follows
\begin{equation}\label{e:083+}
\theta=\left(\prod_{j=1}^m\prod_{i=0}^{\ell_j}\theta_{j,i}\right)^{\frac1{n+1}}\,.
\end{equation}

\begin{proposition}\label{p:2}
Assuming the above definitions we have that $\det G(\vv x)\neq0$ for almost all\, $\vv x\in U$ and we also have that for any\, $\vv x_0\in U$
\begin{equation}\label{e:107+vb}
    \widehat\Theta_{\bmtheta}(\vv x_0)\ \le\ \widehat\Theta:=\max_{1\le
    r\le n}\ \min\limits_{\substack{r_1+\dots+r_m=r\\[0.5ex] 0\le r_j\le \ell_j\ (1\le j\le m) }}\frac{\prod_{j=1}^m\prod_{i=0}^{r_j-1}\theta_{j,i}}{\theta^r}\,.
\end{equation}
In particular, if\/ $\theta\le 1$ and for each $j=1,\dots,d$ ~$\tilde\theta_j$ satisfies Property~\BBG{}
then
\begin{equation}\label{omega2}
    \widehat\Theta\le \max\left\{\frac{\theta_0}{\theta^{n+1}},\frac{1}{\theta_\infty}\right\}\,,
\end{equation}
where
\begin{equation}\label{theta_0}
\theta_0=\min_{1\le j\le m}\theta_{j,0}\qquad\text{and}\qquad \theta_\infty=\max_{\substack{1\le j\le m}}\theta_{j,\ell_j}\,.
\end{equation}
\end{proposition}

\begin{proof}

Let $V=V_1$ be any linear subspace of $\R^{n+1}$ with $\dim V=s\le n$. Let $r=n+1-s$ and let
$r_1,\dots,r_m$ be non-negative integers satisfying $r_1+\dots+r_m=r\le n$ and $0\le r_j\le \ell_j$ $(1\le j\le m)$.

By Lemma~\ref{l3m}, there is a subset $S_1$ of $U_1$ of full Lebesgue measure in $U_1$ such that
$$
\rank\{V_1,\vv f_1(x_1),\dots,\vv f_1^{(r_1)}(x_1)\}=s+r_1
$$
for $x_1\in S_1$. Take any $x_1\in S_1$ and let $V_2$ be the subspace of $\R^{n+1}$ spanned by $V_1$ and $\vv f_1(x_1),\dots,\vv f_1^{(r_1)}(x_1)$. Then, by Lemma~\ref{l3m}, there is a subset $S_{2}(x_1)$ of $U_2$ of full measure such that
$$
\rank\{V_2,\vv f_2(x_2),\dots,\vv f_2^{(r_2)}(x_2)\}=\dim V_2+r_2
$$
for all $x_2\in S_{2}(x_1)$, that is
\begin{equation}\label{lin}
\rank\{V_1,\vv g_1(x_1),\dots,\vv g_1^{(r_1)}(x_1),\vv g_2(x_2),\dots,\vv g_2^{(r_2)}(x_2)\}=s+r_1+r_2
\end{equation}
for all $x_2\in S_{2}(x_1)$. Using Fubini's theorem, one easily checks that the set
$$
S_2:=\{(x_1,x_2):x_1\in S_1,\ x_2\in S_{2}(x_1)\}\subset U_1\times U_2
$$
has full measure in $U_1\times U_2$. By construction, \eqref{lin} holds for all $(x_1,x_2)\in S_2$. Carrying on this procedure by induction in an obvious manner for $i=3,\dots,d$ we will construct a sequence $S_i$ of subsets of $U_1\times \dots \times U_i$ of full Lebesgue measure such that
\begin{equation}\label{lin2}
\rank\{V,\vv f_1(x_1),\dots,\vv f_1^{(r_1)}(x_1),\dots,\vv f_i(x_i),\dots,\vv f_i^{(r_i)}(x_i)\}=s+r_1+\dots+r_i
\end{equation}
for all $(x_1,\dots,x_i)\in S_i$. Taking $i=d$, $V=\{0\}$ and $r_i=\ell_i$ for every $i=1,\dots,d$ \eqref{lin2} implies that $\det G(\vv x)\neq0$ for all $(x_1,\dots,x_m)\in S_m$. Since the set $S_m$ has full Lebesgue measure is $U=U_1\times\dots\times U_m$, it means that $\det G(\vv x)\neq0$ for almost all $\vv x\in U$, thus establishing our first claim within Proposition~\ref{p:2}.

Next, the fact that $S_m$ is of full measure in $U$ implies that $S_m$ is dense in $U$. Then, for any subspace $V$ of $\R^{n+1}$ with $\dim V\le n$ and any $\vv x_0\in U$ there is a point $\vv x\in S_m$ arbitrarily close to $\vv x_0$ such that \eqref{lin2} holds. By \eqref{e:084}, we have that
$$
 \Theta_{\bmtheta}(\vv x,V)\le
\frac{\prod_{j=1}^m\prod_{i=0}^{r_j-1}\theta_{j,i}}{\theta^r}
$$
and since $\vv x$ can be taken arbitrarily close to $\vv x_0$, we get, by the l.h.s. of \eqref{e:085}, that
$$
\widehat\Theta_{\bmtheta}(\vv x_0,V)\le
\frac{\prod_{j=1}^m\prod_{i=0}^{r_j-1}\theta_{j,i}}{\theta^r}\,.
$$
Since the right hand side is not dependent on $V$ as long as $r$ is fixed we have that
$$
\widehat\Theta_{\bmtheta}(\vv x_0,V)\le \min\limits_{\substack{r_1+\dots+r_m=r\\[0.5ex] 0\le r_j\le \ell_j\ (1\le j\le m) }}\frac{\prod_{j=1}^m\prod_{i=0}^{r_j-1}\theta_{j,i}}{\theta^r}\,.
$$
In view of the r.h.s. of \eqref{e:085}, taking the supremum over all $V$ gives \eqref{e:107+vb} and thus proves our second claim within Proposition~\ref{p:2}.

Finally, by the assumption that $\theta\le1$, for
all $r\in\{1,\dots,n\}$ we have $\theta^r\ge
\theta^{n+1}=\theta_0\dots\theta_n$. Therefore
\begin{equation}\label{fg}
\widehat\Theta\le\frac{1}{\theta^{n+1}}\max_{1\le
    r\le n}\ \min\limits_{\substack{r_1+\dots+r_m=r\\[0.5ex] 0\le r_j\le \ell_j\ (1\le j\le m) }}~\prod_{j=1}^m\prod_{i=0}^{r_j-1}\theta_{j,i}\,.
\end{equation}
Let $\ell$ be the number of elements within $\bmtheta$ that are $\le1$. If $r\le \ell$, by Property~M that is satisfied by every $\tilde\theta_j$, we can always choose non-negative integers $r_1,\dots,r_m$ with $r_1+\dots+r_m=r$ such that every factor in the product in \eqref{fg} is $\le1$. Clearly, amongst all $r\le \ell$ the product in \eqref{fg} is maximal when $r=1$. Thus,
\begin{align}
\nonumber \max_{1\le r\le \ell}\ \min\limits_{\substack{r_1+\dots+r_m=r\\[0.5ex] 0\le r_j\le \ell_j\ (1\le j\le m) }}~\prod_{j=1}^m\prod_{i=0}^{r_j-1}\theta_{j,i}
& = \min\limits_{\substack{r_1+\dots+r_m=1\\[0.5ex] 0\le r_j\le \ell_j\ (1\le j\le m) }}~\prod_{j=1}^m\prod_{i=0}^{r_j-1}\theta_{j,i}\\[2ex]
& = \min_{1\le j\le m}\theta_{j,0}\stackrel{\eqref{theta_0}}{=}\theta_0\,.\label{fg2}
\end{align}
If $r>\ell$, then again by Property~M, we can always choose non-negative integers $r_1,\dots,r_m$ with $r_1+\dots+r_m=r$ so that every component of $\bmtheta$ that is $\le1$ will be contained in the product within \eqref{fg}. However, in this case the product will also contain a factor $>1$ (if such components are present in $\bmtheta$). Then amongst all $r$ with $\ell<r\le n$ the product in \eqref{fg} is maximal when $r=n$. Thus,
\begin{align}
\nonumber \max_{\ell< r\le n}\ \min\limits_{\substack{r_1+\dots+r_m=r\\[0.5ex] 0\le r_j\le \ell_j\ (1\le j\le m) }}~\prod_{j=1}^m\prod_{i=0}^{r_j-1}\theta_{j,i}
& = \min\limits_{\substack{r_1+\dots+r_m=n\\[0.5ex] 0\le r_j\le \ell_j\ (1\le j\le m) }}~\prod_{j=1}^m\prod_{i=0}^{r_j-1}\theta_{j,i}\\[2ex]
& = \min_{1\le j\le m}\frac{\theta^{n+1}}{\theta_{j,\ell_j}}~\stackrel{\eqref{theta_0}}{=}~\frac{\theta^{n+1}}{\theta_\infty}\,.\label{fg3}
\end{align}
Combining \eqref{fg2} and \eqref{fg3} together with \eqref{fg} gives \eqref{omega2} and completes the proof.
\end{proof}

\bigskip

Propositions~\ref{p:1} and \ref{p:2} put together imply the following effective result.

\begin{proposition}\label{p:3}
For $j=1,\dots,m$ let $U_j$ be an interval in $\R$, and let $\vv f_j=(f_{j,0},\dots,f_{j,n}):U_j\to\R^{n+1}$ be an $(n+1)$-tuple of real analytic functions linearly independent over $\R$. Let $\ell_i\ge0$ be integers satisfying \eqref{vb1} and for $\vv x=(x_1,\dots,x_m)\in U=U_1\times\dots\times U_m$ let $G(\vv x)$ be given by \eqref{G}.
Then for almost every $\vv x_0\in U$ we have that $\det G(\vv x_0)\not=0$ and there exists a ball $B_0\subset U$ centred at $\vv x_0$ and constants $K_0,\alpha>0$ such that for any ball $B\subset B_0$ there is a constant $\delta=\delta(B,\vv f_1,\dots,\vv f_m)>0$ such that for any $(n+1)$-tuple $\bmtheta$ of positive numbers given by \eqref{theta} such that each $\tilde\theta_j$ satisfies Property~M we have that
\begin{equation}\label{e:086+vb}
\lambda_{\ddd}\Big(B\cap\cA(G,\bmtheta)\Big)\le K_0\,\Big(1+(\widehat\Theta/\delta)^\alpha\Big)\,\theta^{\alpha}\,\lambda_{\ddd}(B)\,,
\end{equation}
where $\widehat\Theta$ is given by \eqref{e:107+vb}.
\end{proposition}

\subsection{Proof of Theorem~\ref{thm-2}}\label{3.3}

We will use the following simple and well known statement.

\begin{lemma}\label{lem-3}
Let $G$ be a $k\times k$ real matrix such that $|\det G|\ge C_1>0$ and every entry of $G$ is bounded above by $C_2>0$. Then, for any column $\vv a\in\R^k$ we have that
$$
|G\vv a|_\infty\ge \frac{C_1}{k!C_2^{k-1}}|\vv a|_\infty,
$$
where $|\cdot|_\infty$ is the supremum norm, that is  $|(y_1,\dots,y_k)|_\infty=\max\{|y_1|,\dots,|y_k|\}$.
\end{lemma}

\begin{proof}
Define the column $\vv b=G\vv a$ and let $a_i$ be the $i$-th coordinate of $\vv a$. Let $G_i$ be the matrix obtained from $G$ by replacing the $i$-th column with $\vv b$. Trivially, we have that $|\det G_i|\le k!C_2^{k-1}|\vv b|_\infty$. Then, by Cramer's rule, we have that
$$
|a_i|=\frac{|\det G_i|}{|\det G|}\le \frac{k!C_2^{k-1}|\vv b|_\infty}{C_1}\,,
$$
whence the required estimate follows.
\end{proof}

\medskip

Now we move onto the proof of Theorem~\ref{thm-2}. As we noted in Remark~\ref{rem3} without loss of generality we will assume that $U_1,\dots,U_m$ are intervals in $\R$. Also without loss of generality we may assume that all the functions $f_{j,i}$ and their derivatives up to the order $n$ are uniformly bounded on $U_1,\dots,U_m$. Otherwise we can simply replace each $U_j$ with a subinterval $U_j'$ arbitrarily close in measure to $U_j$ which closure is contained in $U_j$ and use the standard compactness arguments to enforce the required condition. Thus, we have that
\begin{equation}\label{M}
 M:=\max_{1\le j\le m}\,\max_{0\le \ell\le n+1}\,\max_{0\le i\le n}\,\sup_{x_j\in U_j}|f^{(\ell)}_{j,i}(x_j)|<\infty\,.
\end{equation}
Further, by re-ordering the maps $\vv f_j$ if necessary, without loss of generality we can assume that $v'_1\ge v'_2\ge \ldots\ge v'_m\ge-1$. Let $\phi$ be the largest index such that $v'_\phi$ is strictly bigger than $-1$. The proof splits into two subcases.

\medskip

\noindent$\bullet$ ~~~First consider the case when $m+\phi\ge n+1$. Let $h=n+1-m$ and $G$ be given by \eqref{G} with
$\ell_1=\dots=\ell_h=1$ and $\ell_{h+1}=\dots=\ell_m=0$. Fix any $\vv x\in U$ such that $\det G(\vv x)\neq0$.
Then, by Lemma~\ref{lem-3}, for any integer vector $\vv a\in\Z^{n+1}\setminus\{\vv0\}$ such that \eqref{vb0} is satisfied for the corresponding $F$ we have that
\begin{align*}
\frac{|\det G(\vv x)|}{(n+1)!M^{n}}|\vv a|_\infty &\le |G(\vv x)\vv a|_\infty=\max\Big\{\max_{1\le j\le m}\{|F_j(x_j)|,~\max_{1\le j\le h}\{|F'_j(x_j)|\Big\}\\[1ex]
&\ll H(F)^{\max\{0,-v'_h\}}= |\vv a|_\infty^{\max\{0,-v'_h\}}.
\end{align*}
Hence
$$
|\vv a|_\infty^{1-\max\{0,-v'_h\}}\le \frac{(n+1)!M^{n}}{|\det G(\vv x)|}.
$$
Since $v'_h>-1$, $1-\max\{0,-v'_h\}$ is strictly positive, and the above inequality means that for our fixed $\vv x$ there are only finitely many solutions to \eqref{vb0}. By Proposition~\ref{p:2}, $\det G(\vv x)\neq0$ for almost all\, $\vv x\in U$. Thus, for almost every $\vv x\in U$ system \eqref{vb0} has only finitely many solutions $F\in\cF$. This completes the proof in the case $m+\phi>n+1$.

\bigskip

\noindent$\bullet$ ~~~Now consider the case when $m+\phi< n+1$. Let $G$ be given by \eqref{G} with
$\ell_1=n+2-\phi-m$, $\ell_2=\dots=\ell_\phi=1$ and $\ell_{\phi+1}=\dots=\ell_m=0$.
By Proposition~\ref{p:3}, $\det G(\vv x_0)\neq0$ for almost every point $\vv x_0\in U$. Hence it suffices to prove Theorem~\ref{thm-2} by replacing $U$ with a sufficiently small ball $B_0$ containing $\vv x_0$.
Let $\bmtheta=\bmtheta_t$ be given by \eqref{theta}, where
\begin{align*}
\tilde\theta_1 & =C(2^{-tv_1},2^{-tv_1'},\underbrace{2^t,\dots,2^t}_{n+1-m-\phi\text{ times}})\in\R^{n+3-m-\phi}\,,\\[1ex]
\tilde\theta_i & =C(2^{-tv_i},2^{-tv_i'})\in\R^2\qquad (2\le i\le \phi)\,,\\[3ex]
\tilde\theta_i & =C2^{-tv_i}\in \R\qquad (\phi+1\le i\le m)\,,
\end{align*}
and $C>1$ is a sufficiently large constant. Take any $t\in\N$. Then for any $\vv x\in U$ such that \eqref{vb0} holds for some $F\in\cF$ with $2^t\le H(F)<2^{t+1}$ we have that
\begin{equation}\label{vb0b}
|F(x_j)|< c_j2^{-v_jt}\quad\text{and}\quad
|F'(x_j)|< c'_j2^{-v_j't}
\end{equation}
for all $j\in\{1,\dots,m\}$. By \eqref{M}, we also have that $|F^{(i)}(x_j)|\le 2(n+1)M 2^t$ for any $1\le j\le m$ and $0\le i\le n$. Thus,
by taking
$$
C=\max\{2(n+1)M,c_1,\dots,c_m,c'_1,\dots,c'_m\}
$$
we make sure that $\vv x$ belongs to $\cA(G,\bmtheta_t)$, where the latter is defined in \S\ref{iprp}.
Furthermore, if \eqref{vb0} holds for infinitely many $F\in\cF$, then $\vv x\in \cA(G,\bmtheta_t)$ for infinitely many $t$. By the Borel-Cantelli lemma from probability theory our task of proving Theorem~\ref{thm-2} will be finished if we show that for almost every point $\vv x_0\in U$
\begin{equation}\label{conv1}
  \sum_{t=1}^\infty\lambda_m\big(\cA(G,\bmtheta_t)\cap B_0\big)<\infty
\end{equation}
for a sufficiently small ball $B_0$ centred at $\vv x_0$.

To verify \eqref{conv1} we shall use Proposition~\ref{p:3}. To this end, define
$$
\ve =v_1+\dots+v_m+v_1'+\dots+v_m'-n-1+2m
$$
and note that due to \eqref{ve} we have that $\ve>0$ while due to the definition of $\phi$ we have that
$$
\ve = v_1+\dots+v_m+v_1'+\dots+v_\phi'-n-1+m+\phi\,.
$$
Without loss of generality we can assume that $v_1>\ve$ as otherwise we can make the parameters $v_i$ and $v_i'$ other than $v_1$ smaller to meet this condition while keeping $\ve>0$. This is justifiable because such a change of the parameters would only make the sets $\cA(G,\bmtheta_t)$ bigger.

Using \eqref{e:083} or indeed \eqref{e:083+} observe that
$$
\theta^{n+1}\asymp 2^{-\ve t}.
$$
Further, by \eqref{theta_0} and the above definition of $\bmtheta$, we have that
$$
\theta_0\le \theta_{1,0}\ll 2^{-v_1t}\qquad\text{and}\qquad \theta_\infty\ge \theta_{1,\ell_1}\gg 2^{t}\,.
$$
Observe that in view of the conditions imposed on $v_i$ and $v_i'$ each $\tilde\theta_i$ satisfies property \BBG.
Hence, by \eqref{omega2}, we obtain that
$$
\widehat\Theta\ll \max\{2^{-(v_1-\ve)t},2^{-t}\}~\stackrel{v_1>\ve}{\ll}~ 1.
$$
By Proposition~\ref{p:3}, for almost every point $\vv x_0\in U$ there is a ball $B_0\subset U$ centred at $\vv x_0$ and a constant $\alpha>0$ such that
$$
\lambda_{\ddd}\Big(B_0\cap\cA(G,\bmtheta_t)\Big)\ll 2^{-\ve\alpha t/(n+1)}\quad\text{for all $t\in\N$}\,.
$$
Since the sum $\sum_{t\ge0}2^{-\ve\alpha t/(n+1)}$ converges, \eqref{conv1} follows and the proof is complete.

\section{The convergence case}

\subsection{Auxiliary assumptions and auxiliary statements}

The main and only goal of this section is to prove that $\lambda_d(\cL(\cF,\Psi))=0$ provided that
\begin{equation}\label{conv}
\sum_{h=1}^{\infty}h^{n-m}\Psi^m(h)<\infty.
\end{equation}
Note that the convergence sum condition implies that $h^{n-m}\Psi^m(h)\to0$ as $h\to\infty$, that is
 \begin{equation}\label{eq55}
\Psi(h)=o\Big(h^{-\frac{n-m}m}\Big) \quad \text{as }h\to\infty.
\end{equation}
In view of Remark~\ref{rem3}, we will assume that $U_1,\dots,U_m$ are intervals in $\R$ and so $d=m$. Also as discussed in \S\ref{3.3} we can assume that the constant $M$ defined by \eqref{M} is finite.
Recall that the goal is to prove that the set $\cL(\cF,\Psi)$ of points $(x_1,\dots,x_m)\in U$ such that
\begin{equation}\label{eq40q}
  |F_j(x_j)|<\Psi(H(F))\qquad(1\le j\le m)
\end{equation}
holds for infinitely many $F\in\cF$ has zero Lebesgue measure.

Since the case when $d=1$ is a direct consequence of the main result of \cite{bu1}, for the rest of the proof we will assume that $d\ge2$. For every $k\in\{0,\dots,n\}$ let
$$
\cF_{k}=\big\{F\in \cF: a_k(F)=H(F)\big\}
\qquad\text{and}\qquad
\cF_k(H)=\big\{F\in \cF_k: a_k=H\big\}\,.
$$
Further let $\cL(\cF_k,\Psi)$ consist of points $(x_1,\dots,x_m)\in U$ such that \eqref{eq40} holds for infinitely many $F\in\cF_k$. It is readily seen that
$$
\cL(\cF,\Psi)=\bigcup_{k=0}^n\cL(\cF_k,\Psi)\,.
$$
Thus, it suffices to prove that $\lambda_m(\cL(\cF_k,\Psi))=0$ for all $k$. Since changing the order of the coordinate maps of $\vv f_j$ does not change the properties of $\vv f_j$, namely analyticity and linear independence, it suffices to consider $\cL(\cF_n,\Psi)$ only.

Next, as a consequence of \eqref{M}, we have that for every $F\in\cF$, every $0\le k\le n$, $1\le j\le m$ and $x_j\in U_j$
 \begin{equation}
  \label{eq60b}
|F^{(k)}_j(x_j)|\ll H(F)\,,
 \end{equation}
where the implicit constant only depends on $n$ and $M$.

Since $f_{j,0},\dots,f_{j,n}$ are linearly independent analytic functions, the Wronskian $W(\vv f_j)=\det(f_{j,i}^{(k)})_{0\le i,k\le n}$ is not identically zero. Hence, as a non-zero analytic function, $W(\vv f_j)$ is non-zero everywhere except possibly a countable collection of points. Therefore, without loss of generality we can restrict $\vv x$ to lie in a neighborhood of a points $(x_1',\dots,x_m')\in U$ such all the Wronskians $W(\vv f_j)(x_j')$ are all non-zero. By continuity, making this neighborhood sufficiently small guarantees that $|W(\vv f_j)(x_j)|$ is bounded away from zero for all $\vv x$ in the neighborhood and all $j$. We will assume that this neighborhood is $U$ itself. Thus, for some constant $c_0>0$ we have that
\begin{equation}\label{vb12}
  |W(\vv f_j)(x_j)|\ge c_0\qquad\text{for all $j\in\{1,\dots,m\}$ and $x_j\in U_j$.}
\end{equation}
We will need a couple of other additional assumptions that follow a similar line of argument, namely that
\begin{equation}\label{vb12+++}
\Big|\det(f_{j,i}(x_j))_{\substack{0\le i\le m-1\\[0.3ex] 1\le j\le m~~~}}\Big|\ge c_0\qquad\text{for all }(x_1,\dots,x_m)\in U
\end{equation}
and
\begin{equation}\label{e:043}
\Big|\det(f_{j,i}(x_j))_{\substack{0\le i\le m\\[0.3ex] 1\le j\le m+1~~~}}\Big|\ge c_0\qquad\text{for all }(x_1,\dots,x_m)\in U\,,
\end{equation}
where $\vv f_{m+1}(x_{m+1})$ is any one of $\vv f'_1(x_1)$, \dots, $\vv f'_m(x_m)$.
Conditions \eqref{vb12+++} and \eqref{e:043} are justified by making use of Proposition~\ref{p:2}.

Now we give several lemmas, which are used to obtain upper bounds
for the Lebesgue measure of certain sets. In the first one we let $c/0$ to
be $+\infty$ for any $0<c\le+\infty$.

\begin{lemma}[Lemma~2 in \cite{Ber02}]\label{lem3}
Let $\alpha_0,\dots,\alpha_{N-1},\beta_1,\dots,\beta_N\in\R\cup \{+\infty\}$
be such that $\alpha_0>0$, $\alpha_k>\beta_k\ge 0$ for $k=1,\dots,N-1$ and\/
$0<\beta_N<+\infty$. Let $f:(a,b)\to\R$ be a $C^{(N)}$ function such that
$
\inf_{x\in(a,b)}|f^{(N)}(x)|\ge \beta_N.
$
Then, the set of $x\in(a,b)$ satisfying
\begin{equation}\label{e:008}
\left\{
\begin{array}{l}
\hspace*{8ex}|f(x)|\le \alpha_0,\\[0.5ex]
\ \beta_k \ \le \ |f^{(k)}(x)|\le \alpha_k\ \ (k=1,\dots,N-1)
\end{array}
\right.
\end{equation}
is a union of at most $N(N+1)/2+1$ intervals with lengths at most
$$
\min_{0\le k<l\le N}3^{(l-k+1)/2}(\alpha_k\left/\beta_l\right.)^{1/(l-k)}.
$$
\end{lemma}

\medskip

The following lemma immediately follows from Lemmas~5~and~6 in \cite{BB-Schmidt}. Alternatively, it can be proven by re-using the arguments of \cite[\S5]{Ber02}. The lemma will make use of the constants $M$ and $c_0$, which are defined by \eqref{M} and \eqref{vb12}--\eqref{e:043} respectively.

\medskip

\begin{lemma}\label{lem4}
There exist a constant $\Delta_0=\Delta_0(c_0,M)>0$ such that for every $j\in\{1,\dots,m\}$, for any cube $\cC\subset U$ with side-length $\le \Delta_0$ for any $F\in\cF$ such that \eqref{eq40q} is satisfied for some point $\vv x=(x_1,\dots,x_m)\in \cC$ there exists $\bm N=(N_1,\dots,N_m)\in\{1,\dots,n\}^m$
such that
$$
\inf_{\vv x\in \cC}\,\min_{1\le j\le m}\,|F_j^{(N_j)}(x_j)|\gg H(F)\,,
$$
where the implied constant depends on $c_0$ and $M$ only.
\end{lemma}

\medskip

\begin{remark}
Since $\Delta_0$ in the above lemma is independent of the cube $\cC$, the original domain $U$ can be covered with a certain number of cubes $\cC$ of side-length $\le \Delta_0$. Then $\lambda_m(\cL(\cF_n,\Psi))=0$ will follow on showing that $\lambda_m(\cL(\cF_n,\Psi)\cap\cC)=0$ for every of these cubes. Henceforth, to avoid introducing new unnecessary notation in what follows we simply identify $U$ with one of these cubes. In other words, without loss of generality, we assume that $U$ itself is a cube of sidelength $\le \Delta_0$.
\end{remark}

\medskip

We now have the following useful consequence of the previous remark. For every given vector
$$
\bm N=(N_1,\dots,N_m)\in\{1,\dots,n\}^m
$$
let $\cFN$ be the subcollection of $\cF_n$ such that for every $j=1,\dots,m$
\begin{equation}\label{eq46}
\inf_{x_j\in U_j}\,|F_j^{(N_j)}(x_j)|\gg H(F)\,,
\end{equation}
where the implicit constant is the same as in Lemma~\ref{lem4}. By Lemma~\ref{lem4}, we have that any $F\in\cF_n\setminus \bigcup_{\bm N}\cFN$ simply does not admit any solutions to \eqref{eq40q}. Therefore,
$$
\cL(\cF_n,\Psi)=\bigcup_{\bm N\in\{1,\dots,n\}^m}\cL(\cFN,\Psi)\,,
$$
where, naturally, $\cL(\cFN,\Psi)$ is the set of points $\vv x\in U$ such that \eqref{eq40q} holds for infinitely many $F\in\cFN$. Obviously, to achieve our main goal it suffices to prove that
\begin{equation}\label{vv}
\lambda_m(\cL(\cFN,\Psi))=0
\end{equation}
for any fixed $\bm N$.
%Since the actual value of $\vv N$ will be totally irrelevant, to simplify notation we will use $\cF_n$ to denote $\cF_n^{\bm N}$.

\medskip

The following statement is an immediate consequence of Lemma~\ref{lem3} and \eqref{eq46}.

\begin{lemma}\label{lem5}
There is a constant $K>1$, depending on $c_0$, $M$, $n$ and $\lambda_1(U_j)$ only, such that for any $H\in\N$, any $0\le \beta_1<\alpha_1\le+\infty$ and any $F\in\cF_n(H)$ the set
$$
\sigma_j(F,\alpha_1,\beta_1)=\left\{x_j\in U_j:\begin{array}{l}
                                                |F_j(x_j)|< \Psi(H)\\[0.5ex]
                                                \beta_1\le |F'_j(x_j)|< \alpha_1
                                              \end{array}
\right\}
$$
is the union of at most $K$ intervals $I$ of length
\begin{equation}\label{part1}
|I|\le\frac{\Psi(H)}{\inf\limits_{x_j\in I}|F'_j(x_j)|}\le\frac{\Psi(H)}{\beta_1}\,.
\end{equation}
\end{lemma}

In what follows, the proof of \eqref{vv} will be split into several cases that arise by imposing various additional assumptions.

\subsection{The case $m=n$}

\begin{proposition}\label{p:4}
Suppose that $m=n$ and \eqref{conv} holds. Then, $\lambda_m(\cL(\cF,\Psi))=0$.
\end{proposition}

\begin{proof}
Given $F\in\cF$, let $\sigma(F)$ denote the set of $\vv x=(x_1,\dots,x_m)\in U$ such that inequalities \eqref{eq40q} are satisfied. Let $F\in\cF$ and $\vv a$ be the corresponding integer vector of the coefficients of $F$ and let $\vv x\in\sigma(F)$. Take any $j\in\{1,\dots,m\}$ and let
$$
G(\vv x)=\Big(f_{j,i}(x_j)\Big)_{\substack{0\le i\le m\\[0.3ex] 1\le j\le m+1}}
$$
be the same as in \eqref{e:043}. Recall that $H(F)=|\vv a|_\infty$ and observe that
$G(\vv x)\vv a$ is the column of $F_1(x_1),\dots,F_m(x_m),F_j'(x_j)$.
Then, by Lemma~\ref{lem-3}, we have that
\begin{align*}
\frac{c_0}{(n+1)!M^{n}}H(F) &\le |G(\vv x)\vv a|_\infty=\max\Big\{|F_1(x_1)|,\dots,|F_m(x_m)|,|F'_j(x_j)|\Big\}\\[1ex]
&\le \max\Big\{\Psi(H(F)),|F'_j(x_j)|\Big\}~\stackrel{\eqref{eq55}}{\le}~ \max\Big\{1,|F'_j(x_j)|\Big\}
\end{align*}
for sufficiently large $H(F)$.
Hence
\begin{equation}\label{f1}
\frac{c_0}{(n+1)!M^{n}}H(F)\le |F'_j(x_j)|
\end{equation}
whenever $H(F)\ge H_0$, where $H_0$ is a sufficiently large number depending only on $\Psi$, $c_0$, $n$ and $M$.
Note that in the above argument $j\in\{1,\dots,m\}$ is arbitrary and $\vv x\in\sigma(F)$ is arbitrary. Hence, using \eqref{f1} together with Lemma~\ref{lem5} we conclude that the projection of $\sigma(F)$ on any axis has measure $\ll \Psi(H)H^{-1}$, where $H=H(F)$. Therefore,
$$
\lambda_m(\sigma(F))\ll H^{-m}\Psi(H)^m
$$
provided that $H$ is sufficiently large. Then
$$
\sum_{F\in\cF}\lambda_m(\sigma(F))\ll \sum_{H=H_0}^\infty\sum_{F\in\cF:~H(F)=H}\lambda_m(\sigma(F))\ll
$$
$$
\ll \sum_{H=H_0}^\infty\sum_{F\in\cF:~H(F)=H}H^{-m}\Psi(H)^m\ll \sum_{H=H_0}^\infty H^{n-m}\Psi(H)^m<\infty\,.
$$
By the Borel-Cantelli Lemma, the set of $\vv x\in U$ which belong to infinitely many of $\sigma(F)$ has Lebesgue measure zero. This set is precisely $\cL(\cF,\Psi)$. The proof is thus complete.
\end{proof}

\subsection{The case of big derivatives}

In this subsection we extend the proof of Proposition~\ref{p:4} to the case when all the derivatives $|F'_j(x_j)|$ are relatively large. The underlying idea of the proof originates from \cite{bernik89}.

\begin{proposition}\label{p:5}
Suppose that $m<n$ and \eqref{conv} is satisfied. Then for almost every $(x_1,\dots,x_m)\in U$ there are only finitely many $F\in\cF_n$ simultaneously satisfying \eqref{eq40q} and
\begin{equation}\label{vb1x}
\min_{1\le j\le m}|F'_j(x_j)|\ge H(F)^{\frac12}\,.
\end{equation}
\end{proposition}

\begin{proof}
Let $\ve>0$ be a constant to be chosen later.
Given $F\in\cF_n$, for each $j=1\dots,m$ let
$$
\sigma_j(F)=\Big\{x_j\in U_j:|F_j(x_j)|<\Psi(H(F)),~~|F_j'(x_j)|\ge H(F)^{\frac12}\Big\}\,.
$$
Suppose that $\sigma_j(F)\neq\emptyset$ for all $j\in\{1,\dots,m\}$ and let $x^*_j\in\sigma_j(F)$ be such that
$$
|F_j'(x^*_j)|\le 2\inf_{x_j\in\sigma_j(F)}|F_j'(x_j)|.
$$
Then, by Lemmas~\ref{lem3} and \ref{lem4},
$$
\lambda_1(\sigma_j(F))\ll \frac{\Psi(H)}{|F_j'(x^*_j)|}\le\frac{\Psi(H)}{H^{\frac12}}\,,
$$
where $H=H(F)$. Also let $c_1>0$ be a sufficiently small constant which will be specified later and for each $j=1,\dots,m$ ~let
$$
\tilde\sigma_j(F)=\Big\{x_j\in U_j:|x_j-x^*_j|\le \frac{c_1}{|F'_j(x^*_j)|}\Big\}\,.
$$
Clearly, for $H$ sufficiently large we have that $\lambda_1(\tilde\sigma_j(F))\asymp |F'_j(x^*_j)|^{-1}$ and so
$\lambda_1(\sigma_j(F))\ll \Psi(H)\lambda_1(\tilde\sigma(F))$.
Define
$$
\sigma(F)=\prod_{j=1}^m\sigma_j(F)\qquad\text{and}\qquad \tilde\sigma(F)=\prod_{j=1}^m\tilde\sigma_j(F)\,.
$$
Thus
\begin{equation}\label{11b}
  \lambda_m(\sigma(F))\ll \Psi^m(H)\lambda_m(\tilde\sigma(F))
\end{equation}
provided that $H$ is sufficiently large.

Using Taylor's expansion of $F_j(x_j)$ at $x^*_j$, for any $x_j\in\tilde\sigma_j(F)$ we find that
\begin{align}
\nonumber|F_j(x_j)| &\le |F_j(x^*_j)|+|F'_j(x^*_j)(x_j-x^*_j)|+\tfrac12|F''_j(\tilde x_j)(x_j-x^*_j)^2|\\[1ex]
\nonumber&\ll \Psi(H)+|F'_j(x^*_j)|\cdot\frac{c_1}{|F'_j(x^*_j)|}+H\frac{c_1^2}{|F'_j(x^*_j)|^2}\\[1ex]
&\le \Psi(H)+c_1+c_1^2\ll c_1\label{b12}
\end{align}
for sufficiently large $H$, where the implied constant depends on $M$ and $n$ only.

For each $(n-m)$-tuple $\vv b=(b_{n-1},\dots,b_m)\in\Z^{n-m}$ such that $|b_i|\le H$ for $i=m,\dots,n-1$ define the following subclass of $\cF_{n}(H)$
$$
\cF_{n}(H,\vv b)=\big\{F\in\cF_n(H):a_i(F)=b_i~~\text{if }m\le i\le n-1\big\}\,.
$$
Let $F, T\in \cF_{n}(H,\vv b)$ and assume that $\tilde\sigma(F)\cap\tilde\sigma(T)\neq\emptyset$. Let $\vv x=(x_1,\dots,x_m)$ be any point in this intersection. Define $R=F-T$ and assume that $R\neq 0$. Thus, for some $\vv r=(r_0,\dots,r_{m-1})\in\Z^m\setminus\{\vv0\}$ we have that
$$
R_j(x_j)=r_{0}f_{j,0}(x_j)+\dots+r_{m-1}f_{j,m-1}(x_j)
$$
for $j=1,\dots,m$. By \eqref{b12} applied to both $F$ and $T$, we have that
$$
|R_j(x_j)|=|F_j(x_j)-T_j(x_j)|\le |F_j(x_j)|+|T_j(x_j)|\ll 2c_1
$$
for all $j\in\{1,\dots,m\}$. By Lemma~\ref{lem-3} and inequalities \eqref{vb12+++} and \eqref{M}, we get that
$$
|\vv r|_\infty\ll c_1,
$$
where the implied constant depends only on $M$, $n$ and $c_0$. Hence, there is a sufficiently small choice of $c_1$ determined by $M$, $n$ and $c_0$ only such that for sufficiently large $H$ we have that $|\vv r|_\infty<1$. On the other hand, since $\vv r\in\Z^{m}\setminus\{\vv0\}$ we must have that $|\vv r|_\infty\ge1$. This gives a contradiction, which means that we must have that
\begin{equation}\label{vbvb}
\tilde\sigma(F)\cap\tilde\sigma(T)=\emptyset\qquad\text{for any different }
F, T\in \cF_{n}(H,\vv b)
\end{equation}
provided that $H$ is sufficiently large and $c_1>0$ is sufficiently small.
Hence
\[
\sum_{F\in \cF_{n}(H,\vv b)}~~\lambda_m(\tilde\sigma(F))\le \lambda_m(U)<\infty.
\]
Together with \eqref{11b} this gives that
\[
\sum_{F\in \cF_{n}(H,\vv b)}~~\lambda_m(\sigma(F))\ll\Psi^m(H),
\]
which further implies that
\[
\sum_{F\in\cF_n}\lambda_m(\sigma(F))=\sum_{H=1}^{\infty}~\sum_{\vv b\in\Z^{n-m}}~\sum_{F\in \cF_{n}(H,\vv b)}~\lambda_m(\sigma(F))\ll \sum_{H=1}^{\infty}H^{n-m}\Psi^m(H)<\infty.
\]
Using the Borel-Cantelli lemma now completes the proof.
 \end{proof}

\subsection{The case of `too good' approximations}

In this subsection we deal with several instances when one (or two) derivatives $|F'_j(x_j)|$ happen to be too small or $\Psi(H)$ is much smaller than the upper bound we have from \eqref{eq55}. In either case the extra approximation property will enable us to either appeal to Theorem~\ref{thm-2} or simply use the Borel-Cantelli lemma.

We begin with the observation that Proposition~\ref{p:5} effectively allows us impose the condition that $|F'_j(x_j)|<H(F)^{\frac12}$ for some $j$. The following statement deals with the case when another derivative also happens to be small.

\medskip

\begin{proposition}\label{p:7}
Let $2\le m<n$, $1\le j_1\neq j_2\le m$ be integers and \eqref{conv} be satisfied. Then, for any $\delta>0$ for almost every $(x_1,\dots,x_m)\in U$ there are only finitely many $F\in\cF_n$ simultaneously satisfying \eqref{eq40q},
\begin{equation}\label{eq60++}
  |F'_{j_1}(x_{j_1})|< H(F)^{\frac12-\delta}\quad\text{and}\quad  |F'_{j_2}(x_{j_2})|< H(F)^{\frac12}\,.
 \end{equation}
\end{proposition}

\begin{proof}
Take any $\vv x=(x_1,\dots,x_m)\in U$ such that \eqref{eq40q} and \eqref{eq60++} are satisfied for infinitely many $F\in\cF_n$. By \eqref{eq55} and \eqref{eq60b}, we then have that \eqref{vb0} is satisfied for infinitely many $F\in\cF$ with the following choice of exponents
\[
 \begin{array}{l}
v_1=\dots=v_m=\frac{n-m}m,~~~~~v'_{j_1}=-\tfrac12+\delta,~v'_{j_2}=-\tfrac12,\\[2ex]
 v'_j=-1\quad\text{for }j\ne j_1,j_2~~~(1\le j\le m).
 \end{array}
\]
It is easy to see that condition \eqref{ve} holds. Then, applying Theorem~\ref{thm-2} completes the proof of the proposition.
\end{proof}

\medskip

Now we deal with the case of `small' $\Psi$. When dealing with this case we will naturally assume that one of the derivatives $|F_j(x_j)|$ is small, namely, we will assume the inequality opposite to \eqref{vb1x}, as otherwise we are covered by Proposition~\ref{p:5}.

\begin{proposition}\label{p:8}
Suppose that $m<n$ and \eqref{conv} is satisfied. Then for almost every $(x_1,\dots,x_m)\in U$ there only finitely many $H\in\N$ such that $\Psi(H)< H^{-\frac{n+\frac12-m}{m}-\delta}$ and for some $F\in\cF_n(H)$ inequalities  \eqref{eq40q} and
\begin{equation}\label{vb1x=}
\min_{1\le j\le m}|F'_j(x_j)|< H^{\frac12}
\end{equation}
are simultaneously satisfied.
\end{proposition}

\begin{proof}
Once again the proof is readily obtained by applying Theorem~\ref{thm-2}, this time with the following choice of exponents:
\[
 \begin{array}{l}
v_1=\dots=v_m=\frac{n+\frac12-m}m+\delta,~~~~~v'_{j_1}=-\tfrac12,\\[2ex]
 v'_j=-1\quad\text{for }j\ne j_1~~~(1\le j\le m),
 \end{array}
\]
where $1\le j_1\le m$ is fixed but arbitrary.
\end{proof}

\bigskip

\begin{proposition}\label{p:6}
Let $\bm N=(N_1,\dots,N_m)\in\{1,\dots,n\}^m$. Then for any $1\le j\le m$ and any $0\le k_j< N_j$ for almost all $\vv x=(x_1,\dots,x_m)\in U$ there are only finitely many $F\in\cFN$ such that
\begin{equation}\label{eq60}
  |F^{(k_j)}_j(x_j)|< H(F)^{-n(n+1)}\,.
 \end{equation}
\end{proposition}

\begin{proof}
By \eqref{eq46} and Lemma~\ref{lem3}, for any fixed $F\in\cFN$ the set $\sigma_j(F)$ of $x_j\in U_j$ satisfying \eqref{eq60} has measure $\ll H(F)^{-(n(n+1)+1)/n}=H(F)^{-n-1-1/n}$. Since the number of $F\in\cF_n$ with $H(F)=H$ is $\ll H^n$ we get that
$$
\sum_{F\in\cFN}\lambda_1(\sigma_j(F))\ll\sum_{H=1}^\infty H^n\cdot H^{-n-1-1/n}
=\sum_{H=1}^\infty H^{-1-1/n}<\infty.
$$
Hence, by the Borel-Cantelli Lemma, any $x_j$ that belongs to infinitely many of $\sigma_j(F)$ lies in a subset of $U_j$ of Lebesgue measure zero, say $S_j$. Therefore, using Fubili's theorem we conclude that any $\vv x=(x_1,\dots,x_m)\in U$ such that $x_j$ lies in $S_j$ has $\lambda_m$-measure zero and the proof is thus complete.
\end{proof}

\subsection{The remaining case}

From now on we fix any $\bm N=(N_1,\dots,N_m)\in\{1,\dots,n\}^m$ and let $\delta$ be a small positive real number that will be specified later.
Given any collection of integers $\bm\ell=\{\ell_0,\ell_{j,i}:1\le j\le m,~1\le i\le N_j\}$, define $\cL(\cFN,\Psi,\bm\ell)$ to be the set of $(x_1,\dots,x_m)\in U$ such that
for infinitely many $H\in\N$ with
\begin{equation}\label{Psiv}
H^{-(\ell_0+1)\delta}\le\Psi(H)< H^{-\ell_0\delta}
\end{equation}
there exists $F\in\cF_n(H)$ satisfying \eqref{eq40q} and
\begin{equation}\label{eq87}
  H(F)^{(\ell_{j,i}-1)\delta}\le |F^{(i)}_{j}(x_{j})|< H(F)^{\ell_{j,i}\delta}
 \end{equation}
for all $1\le j\le m$ and $1\le i<N_j$.

In view of Propositions~\ref{p:4}--\ref{p:6}, to complete the proof of the convergence case of Theorem~\ref{t2} it remains to show that $\lambda_m(\cL(\cFN,\Psi,\bm\ell))=0$ in the following instances:
\begin{align}
\label{l-1}\tfrac{n-m}{m}&\le \ell_0\delta \le \tfrac{n+\frac12-m}{m}\\[1.5ex]
\label{l-2}-n(n+1)&\le \ell_{j_0,1}\delta\le \tfrac12&&\text{for some $j_0$ }(1\le j_0\le m),\\[2ex]
\label{l-3}\tfrac12 &\le \ell_{j,1}\delta \le 1+\delta&&\text{for all $j\neq j_0$ $(1\le j\le m)$}\\[2ex]
\label{l-4}-n(n+1)&\le \ell_{j,i}\delta\le 1+\delta&&\text{for all }1\le j\le m,~2\le i< N_j.
\end{align}
Recall that we agreed that $m\ge2$ as the case $m=1$ is done in \cite{bu1}.

\bigskip

\begin{proposition}\label{p:9}
Suppose that $2\le m<n$ and \eqref{conv} is satisfied. Then for any collection of integer parameters $\bm\ell$ as above subject to \eqref{l-1}--\eqref{l-4} we have that $\lambda_m(\cL(\cFN,\Psi,\bm\ell))=0$.
\end{proposition}

\begin{proof}
Without loss of generality we assume that \eqref{l-2} is satisfied with $j_0=1$ and hence \eqref{l-3} is satisfied for all $j$ $(2\le j\le m)$. Fix any $\bm\ell$ satisfying \eqref{l-1}--\eqref{l-4}. Let $H\in\N$ be such that \eqref{Psiv} is satisfied. Given
$$
F\in\cFN(H):=\big\{F\in\cFN:H(F)=H\}\,,
$$
for each $j=1\dots,m$ ~let
$$
\sigma_j(F)=\Big\{x_j\in U_j:|F_j(x_j)|<\Psi(H)~~\&~~\eqref{eq87}\text{ hold for all $1\le i<N_j$.}
\Big\}\,.
$$
Then, by Lemma~\ref{lem5}, there is a constant $K>1$ such that any non-empty $\sigma_j(F)$ is the union of at most $K$ intervals, say $\sigma_j(F,t_j)$. We have that $1\le t_j\le K$ although the actual number of intervals may be less than $K$. For each of these intervals fix a point $\kappa_j=\kappa_j(F,t)\in\sigma_j(F,t_j)$ such that
$$
|F_j'(\kappa_j)|\le 2\inf_{x_j\in\sigma_j(F,t_j)}|F_j'(x_j)|.
$$
Then, by Lemmas~\ref{lem5},
\begin{equation}\label{r4}
\lambda_1(\sigma_j(F,t_j))\ll \frac{\Psi(H)}{|F_j'(\kappa_j)|}\le\frac{\Psi(H)}{H^{(\ell_{j,1}-1)\delta}}\,,\qquad\kappa_j=\kappa_j(F,t_j).
\end{equation}
Further for $j=2,\dots,m$ we introduce the following auxiliary intervals:
$$
\tilde\sigma_j(F,t_j)=\Big\{x_j\in U_j:|x_j-\kappa_j|\le \frac{(\Psi(H)\, H^{n-m})^{-\frac1{m-1}}}{|F'_j(\kappa_j)|}\Big\},\quad\text{where }\kappa_j=\kappa_j(F,t_j).
$$
Also define $\tilde\sigma_1(F,t_j)=\sigma_1(F,t_j)$, and finally, given a $\bm t=(t_1,\dots,t_m)$ ($t_j\le K$), let
$$
\sigma(F,\bm t)=\prod_{j=1}^m\sigma_j(F,t_j)\qquad\text{and}\qquad \tilde\sigma(F,\bm t)=\prod_{j=1}^m\tilde\sigma_j(F,t_j)
$$
as long as each $\sigma_j(F,t_j)$ is defined. Clearly, each set $\sigma(F,\bm t)$ and $\tilde\sigma(F,\bm t)$ is a rectangle in $U$. In a nutshell, for a fixed $F$ the collection of all possible sets $\sigma(F,\bm t)$ as $\bm t$ varies, is simply the decomposition of $\prod_{j=1}^m\sigma_j(F)$ into rectangles. Our immediate goal is to show that $\tilde\sigma(F,\bm t)$ is an `expansion' of $\sigma(F,\bm t)$.

To this end, note that, by \eqref{eq55}, as $H\to\infty$ we have that
$\Psi(H)^m \le o( H^{-(n-m)})$,
which trivially implies that $\Psi(H)^{m-1}=o\big((\Psi(H)\, H^{n-m})^{-1}\big)$ and
further gives that
$$
\Psi(H)=o\Big((\Psi(H)\, H^{n-m})^{-\frac1{m-1}}\Big)\qquad\text{as }H\to\infty.
$$
Therefore, for $2\le j\le m$ we get that $\lambda_1(\sigma_j(F,t_j))=o(\lambda_1(\tilde\sigma_j(F,t_j)))$ as $H\to\infty$. Therefore, using the fact that $\kappa_j\in\sigma_j(F,t_j)$ we conclude that
$$
\sigma_j(F,t_j)\subset \tilde\sigma_j(F,t_j)\qquad\text{for sufficiently large $H$}.
$$
The above is also trivially true for $j=1$. Hence
\begin{equation}\label{dd}
\sigma(F,\bm t)\subset \tilde\sigma(F,\bm t)\qquad\text{for sufficiently large $H$}.
\end{equation}
Finally, note that for sufficiently large $H$
$$
\lambda_1(\tilde\sigma_j(F,t_j))\asymp \frac{(\Psi(H)\, H^{n-m})^{-\frac1{m-1}}}{|F'_j(\kappa_j)|}
$$
and therefore, by \eqref{r4}, we get that
$$
\frac{\lambda_1(\sigma_j(F,t_j))}{\lambda_1(\tilde\sigma_j(F,t_j))}\ll \frac{\Psi(H)}{(\Psi(H)\, H^{n-m})^{-\frac1{m-1}}}\qquad(2\le j\le m).
$$
Henceforth,
\begin{equation}\label{11b+}
  \lambda_m(\sigma(F,\bm t))\ll H^{n-m}\Psi^m(H)\lambda_m(\tilde\sigma(F,\bm t))
\end{equation}
provided that $H$ is sufficiently large.

We will call $\sigma(F,\bm t)$ an \emph{essential domain}\/ if $\tilde\sigma(F,\bm t)\cap\tilde\sigma(\hat F, \hat{\bm t})=\emptyset$ for any other $\hat F\in\cFN(H)$ and any possible $\hat{\bm t}$. Otherwise we will call $\sigma(F,\bm t)$ \emph{inessential}. Note that if $\sigma(F,\bm t)$ is essential, we allow that $\tilde\sigma(F,\bm t)$ is intersected by another domain $\tilde\sigma(F, \hat{\bm t})$ for the same $F$ and a different $\hat{\bm t}$. However the multiplicity of such intersections in at most the number of possible $m$-tuples $\bm t$, that is $K^m$. Therefore, we have that
$$
\sum_{\text{essential $\sigma(F,\bm t)$}}\lambda_m(\tilde\sigma(F,\bm t))\le K^m\lambda_m(U).
$$
Hence, by \eqref{11b+} we get that
$$
\sum_{H\in\N:\text{ \eqref{Psiv} holds}}~~\sum_{\substack{F\in\cFN, ~H(F)=H\\[0.5ex]\text{$\sigma(F,\bm t)$ is essential}}}\lambda_m(\sigma(F,\bm t))\ll\sum_{H=1}^\infty H^{n-m}\Psi^m(H)<\infty.
$$
By the Borel-Cantelli lemma, the set of $\vv x$ that fall into infinitely many essential domains is of Lebesgue measure zero.

\medskip

Now suppose that $\vv x=(x_1,\dots,x_m)\in U$ belongs to an infinite number of inessential domains and let $\sigma(F,\bm t)$ be one of these domains. By \eqref{dd}, we have that
$\vv x\in\tilde\sigma(F,\bm t)$ while by the definition of inessential domains we have that
$\tilde\sigma(F,\vv t)\cap\tilde\sigma(\hat F,\hat{\bm t})\neq\emptyset$ for some $\hat F\in\cFN(H)$ different from $F$ and some $\hat{\bm t}$. In particular, it implies that
\begin{equation}\label{nj}
  |x_j-\kappa_j|\ll \frac{(\Psi(H)\, H^{n-m})^{-\frac1{m-1}}}{H^{(\ell_{j,1}-1)\delta}}\le \frac{(H^{-(\ell_0-1)\delta}\, H^{n-m})^{-\frac1{m-1}}}{H^{(\ell_{j,1}-1)\delta}}\qquad(2\le j\le m),
\end{equation}
where $\kappa_j$ is any of $\kappa_j(F,\bm t)$ and $\kappa_j(\hat F,\hat{\bm t})$.

Now we use Taylor's expansion of $F_j(x_j)$ and $\hat F_j(x_j)$ at $\kappa_j$ $(2\le j\le m)$ to estimate $|F_j(x_j)|$ and $|\hat F_j(x_j)|$. Then, using \eqref{eq60b}, \eqref{eq87}, \eqref{l-3}, \eqref{nj}, Taylor's formula and the fact that $\kappa_j\in\sigma_j(F,t_j)$, we get
\begin{align}
\nonumber|F_j(x_j)| &\le |F_j(\kappa_j)|+|F'_j(\kappa_j)(x_j-\kappa_j)|+\tfrac12|F''_j(\tilde x_j)(x_j-\kappa_j)^2|\\[1ex]
\nonumber&\ll \Psi(H)+
H^{\ell_{j,1}\delta}\cdot\frac{(\Psi(H)\, H^{n-m})^{-\frac1{m-1}}}{H^{(\ell_{j,1}-1)\delta}}+
H\left(\frac{(\Psi(H)\, H^{n-m})^{-\frac1{m-1}}}{H^{(\ell_{j,1}-1)\delta}}\right)^2\\[1ex]
\nonumber&~\stackrel{\eqref{eq87}\&\eqref{l-3}}{\ll}~ \Psi(H)+H^{\delta}(\Psi(H)\, H^{n-m})^{-\frac1{m-1}}+H\left(\frac{(\Psi(H)\, H^{n-m})^{-\frac1{m-1}}}{H^{\frac12-\delta}}\right)^2\\[1ex]
&=~ \Psi(H)+H^\delta(\Psi(H)\, H^{n-m})^{-\frac1{m-1}}\left(1+H^\delta(\Psi(H)\, H^{n-m})^{-\frac1{m-1}}\right).\label{e62}
\end{align}
By \eqref{eq55},
$$
(\Psi(H)\, H^{n-m})^{-\frac1{m-1}}\gg (H^{-\frac{n-m}{m}}\, H^{n-m})^{-\frac1{m-1}}=
H^{-\frac{n-m}{m}}\gg \Psi(H)\,.
$$
Hence, on imposing the condition that $\delta<1/m$, we have, by \eqref{Psiv}, \eqref{e62} and the assumption $n>d$, that
\begin{align}
  \nonumber|F_j(x_j)|&\ll H^{2\delta}(\Psi(H)\, H^{n-m})^{-\frac1{m-1}}\\[1ex]
  \nonumber&\le  H^{2\delta}(H^{-(\ell_0+1)\delta} H^{n-m})^{-\frac1{m-1}} \\[1ex]
  & = H^{-\frac{n-m-\ell_0\delta-\delta(2m-1)}{m-1}}\qquad(2\le j\le m)\,. \label{k1}
\end{align}
Regarding $\hat F$ we have similarly that
\begin{align}
|\hat F_j(x_j)|\ll H^{-\frac{n-m-\ell_0\delta-\delta(2m-1)}{m-1}}\qquad(2\le j\le m)\,. \label{k2}
\end{align}

Now, since $\tilde\sigma_1(F,t_1)=\sigma_1(F,t_1)$ and $x_1\in \sigma_1(F,t_1)$, we plainly have that
\begin{equation}\label{k2+}
  |F_1(x_1)|\le \Psi(H)\le H^{-\ell_0\delta},\qquad
  |F'_1(x_1)|\le H^{\ell_{1,1}\delta}.
\end{equation}
Finally, we estimate $|\hat F_1(x_1)|$ and $|\hat F'_1(x_1)|$. At this point we need the following improved estimate for the size of $\sigma_1(F,t_1)$, which follows from Lemma~\ref{lem3}, \eqref{eq40q}, \eqref{eq46}, \eqref{Psiv} and \eqref{eq87}:
\begin{equation}\label{hj}
|\sigma_1(F,t_1)|\ll
\min\left\{\min_{1\le i\le N_1}H^{-\frac{(\ell_0+\ell_{1,i}-1)\delta}{i}},~\min_{2\le i\le N_1}H^{\frac{(\ell_{1,1}-\ell_{1,i}+1)\delta}{i-1}}\right\},
\end{equation}
where $\ell_{N_1}$ is formally defined to satisfy $\ell_{N_1}\delta=1$ (not necessarily an integer).
Since $\sigma_1(F,t_1)\cap\sigma_1(\hat F,\hat t_1)\neq\emptyset$, there is a point, say $y_1$ in this intersection. Since $\sigma_1(F,t_1)$ is an interval, we have that $|x_1-y_1|\le |\sigma_1(F,t_1)|$. Hence, using Taylor's formula, we obtain the following estimate:
\begin{align}
\nonumber|\hat F'_1(x_1)|&\ll\sum_{i=1}^{N_1}|\hat F_1^{(i)}(y_1)|\cdot|x_1-y_1|^{i-1}\\[0ex]
\nonumber&\stackrel{\eqref{eq87}}{\ll}~\sum_{i=1}^{N_1}H^{\ell_{1,i}\delta}\cdot|\sigma_1(F,t_1)|^{i-1}\\[0ex]
\nonumber&\stackrel{\eqref{hj}}{\ll}~\sum_{i=1}^{N_1}H^{\ell_{1,i}\delta}\cdot
H^{(\ell_{1,1}-\ell_{1,i}+1)\delta}\\[0ex]
&\ll ~H^{(\ell_{1,1}+1)\delta}\,.\label{k3}
\end{align}
Similarly, we get that
\begin{align}
\nonumber|\hat F_1(x_1)|&\ll\sum_{i=0}^{N_1}|\hat F_1^{(i)}(y_1)|\cdot|x_1-y_1|^{i}\\[0ex]
\nonumber&\stackrel{\eqref{eq87}}{\ll}~\sum_{i=1}^{N_1}H^{\ell_{1,i}\delta}\cdot|\sigma_1(F,t_1)|^{i}\\[0ex]
\nonumber&\stackrel{\eqref{hj}}{\ll}~\sum_{i=1}^{N_1}H^{\ell_{1,i}\delta}\cdot
H^{-(\ell_{0}+\ell_{1,i}-1)\delta}\\[0ex]
&\ll ~H^{-(\ell_{0}-1)\delta}\,.\label{k4}
\end{align}

Now define $R=F-\hat F$. Since $F$ and $\hat F$ share the same coefficient of $a_n$, namely $H$, it is cancelled in the difference and we get that
$$
R_j(x_j)=b_{0}f_{j,0}(x_j)+\dots+b_{n-1}f_{j,n-1}(x_j)\qquad(1\le j\le m),
$$
where  $\vv b=(b_0,\dots,b_{n-1})\in\Z^{n}\setminus\{\vv0\}$ and $|\vv b|_\infty\le 2H$. By \eqref{k1}, \eqref{k2}, \eqref{k2+}, \eqref{k3} and \eqref{k4} we get that
\begin{align}
 |R_1(x_1)| & \ll ~H^{-(\ell_0-1)\delta},\label{w1}\\[1ex]
 |R'_1(x_1)| & \ll H^{(\ell_{1,1}+1)\delta},\label{w2}\\[1ex]
 |R_j(x_j)| & \ll H^{-\frac{n-m-\ell_0\delta-\delta(2m-1)}{m-1}}\qquad(2\le j\le m)\label{w3}\,.
\end{align}
Since, by our assumption, $\vv x$ lies in infinitely many inessential domains, \eqref{w1}--\eqref{w3} hold for infinitely many $H$ and some $R$ as above. If $x_1$ s a root of some $R$ like that, then it must lie in a countable set (this set is made of roots of a countable family of analytic functions). Therefore, $\vv x$ lies in a set of measure zero. Otherwise, we must have infinitely many different $R$ satisfying \eqref{w1}--\eqref{w3}. In this remaining case we can once again appeal to Theorem~\ref{thm-2} with $\cF$ being replaced by the collection of all the maps $R$ defined above and exponents
\[
 \begin{array}{l}
v_1=\ell_0\delta-\delta,\qquad\qquad~~ v'_1=-\ell_{1,1}\delta-\delta,\\[1ex]
v_j=\frac{n-m-\ell_0\delta-\delta(2m-1)}{m-1},\quad v'_j=-1\qquad (2\le j\le m).
 \end{array}
\]

Since $n$ becomes smaller by one (due to the cancellation of $a_n$'s), condition \eqref{ve} will read as follows
$$
v_1+\dots+v_m+v'_1+\dots+v'_m>n-2m
$$
and is the only remaining thing to justify the use of Theorem~\ref{thm-2}. To verify it we use \eqref{l-2} (recall that $j_0=1$):
\begin{align*}
v_1+&\dots+v_m+v'_1+\dots+v'_m\\[1ex]
&=\ell_0\delta-\delta+n-m-\ell_0\delta-\delta(2m-1)-\ell_{1,1}\delta-\delta-(m-1)\\[1ex]
&=n-2m+1-(2m+1)\delta-\ell_{1,1}\delta\\[1ex]
&\ge n-2m+\tfrac12-(2m+1)\delta>n-2m
\end{align*}
provided that $\delta<\frac{1}{2(2m+1)}$. Thus, the use of Theorem~\ref{thm-2} is justified and we conclude that
$\vv x$ lies in a set of measure zero. This completes the proof of Proposition~\ref{p:9} and completes the proof of the convergence case of Theorem~\ref{t2}.
\end{proof}

\section{The divergence case}\label{sec4}

Here we prove the following more general divergence result for Hausdorff measures. In what follows,
by a dimension function $g$ we mean a continuous monotonically increasing function defined on $(0,+\infty)$ such that $\lim_{r\to0^+}g(r)=0$. Also, $\cH^g$ will denote the $g$-dimensional Hausdorff measure -- see \cite{BDV06} and references within for further details.

\begin{theorem}\label{t2+}
Let $n\ge m\ge1$ be integers.
Let $U=U_1\times\dots\times U_m\subset\R^d$, $\vv f_1,\dots,\vv f_m$, $\cF=\cF(\vv f_1,\dots,\vv f_m)$, $\Psi$ and $\cL(\cF,\Psi)$ be as in Theorem~\ref{t2} and $\Psi$ be monotonic. Suppose that for each $j=1,\dots,m$ the coordinate functions $f_{j,0},\dots,f_{j,n}$ of the map $\vv f_j$ are analytic and linearly independent over $\R$. Further let $g$ be any dimension function such that $\tilde g(r):=r^{-d+m}g(r)$ is increasing and $r^{-m}\tilde g(r)$ is non-increasing. Then
\begin{equation}\label{vb90}
  \cH^g(\cL(\cF,\Psi))=\cH^g(U) \qquad \text{if ~~$\sum_{h=1}^\infty h^n\,\tilde g\Big(\frac{\Psi(h)}{h}\Big)=\infty$}\,.
\end{equation}
\end{theorem}

\medskip

It is well known that for any Lebesgue measurable set $X\subset\R^d$ and $g(r)=r^d$ we have that $\cH^g(X)=c_d\lambda_d(X)$, where $c_d$ is a fixed positive constant. Hence, as is easily seen, Theorem~\ref{t2+} contains the divergence case of Theorem~\ref{t2}. Furthermore, Theorem~\ref{t2+} is also applicable to the case of convergence within Theorem~\ref{t2}, giving us a non-trivial lower bound on the size of the corresponding null sets. Indeed, applying Theorem~\ref{t2+} with $g(r)=r^s$ for some $s>0$ gives the following corollary regarding the Hausdorff dimension of $\cL(\cF,\Psi)$.

\begin{corollary}\label{cor1}
Under the same conditions as in Theorem~\ref{t2+}, we have that
\begin{equation}\label{lb}
\dim \cL(\cF,\Psi)\ge \min\Big\{d,~\frac{n+1}{\tau_\Psi+1}+d-m\Big\}\,,
\end{equation}
where
$$
\tau_\Psi=\liminf_{h\to\infty}\frac{-\log\Psi(h)}{\log h}
$$
is the lower order of $1/\Psi$ at infinity. In particular, if $\Psi_\tau(h)=h^{-\tau}$ with $\tau>\frac{n+1-m}{m}$ we have that
$$
\dim \cL(\cF,\Psi_\tau)\ge \frac{n+1}{\tau+1}+d-m\,.
$$
\end{corollary}

\bigskip

\begin{remark}
In the case of one linear form, that is $m=1$, the above corollary follows from a result of Dickinson and Dodson \cite{DD00}. In the case $m=1$, $n=2$ the corresponding upper bound was found by Baker \cite{RBaker}. More recently Huang \cite{Huang} proved a more precise version of Baker's result for Hausdorff measures that involves the convergence of a sum as in \eqref{vb90}. Also in the case of polynomials, that is when $m=1$, $d=1$, $n\ge2$ and $\vv f(x)=(1,x,\dots,x^n)$, the lower bound \eqref{lb} is a result of Baker and Schmidt \cite{BS}. Furthermore, in this latter case we also have the corresponding upper as a result of \cite{Bernik}. In the general case, establishing upper bounds complementary to \eqref{lb} remains a challenging open problem. Part of this problem is to prove the following
\end{remark}

\bigskip

\noindent\textbf{Conjecture:}
Under the same conditions as in Theorem~\ref{t2+}, one should have that
$$
\dim \cL(\cF,\Psi)=\min\Big\{d,~\frac{n+1}{\tau_\Psi+1}+d-m\Big\}\,.
$$

\subsection{Ubiquity}

In this subsection we discuss the concept of Ubiquity defined in \cite{BDV06} and state a key lemma regarding ubiquitous systems that will be instrumental in the proof of Theorem~\ref{t2+}. First, we recall the basic definitions from \cite{BDV06} in a simplified form necessary for the application that we have in mind. In what follows:

\begin{itemize}
  \item $\Omega$ is a closed ball in $\R^{m}$;
  \item $\cR:=(\ra)_{\al\in J}$ is a family of points $\ra$ in $\Omega$
  (usually referred to as \emph{resonant points}) indexed by a countable set $J$;
  \item $\beta:J\to \Rp:\alpha\mapsto\ba$ is a function on $J$, which attaches a `weight' $\ba$ to resonant points $\ra$;
  \item $J(t):=\{\al \in J:\ba\le 2^t\}$ is assumed to be finite for any $t\in\N$;
  \item $\rho: \Rp \to\Rp$ is a function such that $\lim\limits_{r\to\infty}\rho(r)=0$ referred to as the {\em ubiquity function};
  \item $B(\vv x,r)$ is a ball in $\Omega$ centred at $\vv x\in\Omega$ of radius $r>0$ defined using the supremum norm. Note that, by definition, $B(\vv x,r)$ consists of points in $\Omega$ only.
\end{itemize}

\begin{definition}\label{rs}\label{US}\rm
The pair $(\cR;\beta)$ is called
{\em a locally ubiquitous system in $\Omega$ relative to $\rho$} if there is an
absolute constant $k_0>0$ such that for any ball $B$ in $\Omega$
\begin{equation}\label{e:018}
\lambda_{m}\Big(\,\bigcup_{\al\in J(t)}
B\big(\ra,\r(2^t)\big)\cap B\Big) \ \ge \  k_0 \, \lambda_m(B)
\end{equation}
for all sufficiently large $t$.
\end{definition}

\medskip

Given a function $\Phi:\Rp\to\Rp$, let
$$
\Lambda_\cR(\Phi) \ := \ \{\vv x\in \Omega:|\vv x-\ra|_\infty<\Phi(\ba)
\ \mbox{holds for\ infinitely\ many\ }\al \in J \} \,.
$$
The following lemma follows from Theorems~1 and 2 of
\cite{BDV06} (the parameter $\gamma$ should be taken to be $0$ in these theorems) and can also be found as Theorem~1 in \cite{Paris}.

\begin{lemma}\label{l:01}
Let $\Phi:\Rp\to\Rp$ be a monotonic function and $\Omega,J,\cR,\beta,\rho$ be as above. Suppose that $(\cR,\beta)$ be a locally ubiquitous system in $\Omega$ relative to $\rho$. Let $\tilde g$ be a dimension function such that $r^{-m}\tilde g(r)$ is non-increasing. Suppose further that
\begin{equation}\label{reg}
\limsup_{t\to\infty}  \frac{\rho(2^{t+1})}{\rho(2^t)}<1\,.
\end{equation}
Then
\begin{equation}\label{e:019}
 \cH^{\tilde g}\big(\Lambda_\cR(\Phi)\big) \
= \
\cH^{\tilde g}(\Omega)\qquad\text{if}\qquad\sum_{t=1}^{\infty}\frac{\tilde g(\Phi(2^t))}{\rho(2^t)^{\ddd}}\
= \ \infty \, .
\end{equation}
\end{lemma}

\bigskip

We now establish a specific example of a ubiquitous system that will be used in the proof of Theorem~\ref{t2+}.

\begin{proposition}\label{ublem}
Let $n\ge m\ge1$ be integers, $d_1=\dots=d_m=1$ so that $U_1$,\dots,$U_m$ are intervals in $\R$.
Let $U=U_1\times\dots\times U_m\subset\R^m$, $\vv f_1,\dots,\vv f_m$, $\cF=\cF(\vv f_1,\dots,\vv f_m)$ be as in Theorem~\ref{t2+} and suppose that \eqref{M} is satisfied. Then, for almost every $\vv x_0\in U$ and any fixed $0<\delta_0\le 1$ there exists a closed ball $\Omega\subset U$ centred at $\vv x_0$ and a constant $\eta>0$ such that $(\cR,\beta)$ is locally ubiquitous in $\Omega$ relative to $\rho$, where
\begin{align*}
&J=\big\{(F,\gamma_1,\dots,\gamma_m)\in\cF\times \Omega~:~F_1(\gamma_1)=\dots=F_m(\gamma_m)=0\big\}\,,\\[1ex]
&\cR=\big\{R_\alpha=(\gamma_1,\dots,\gamma_m)\in \Omega:\alpha=(F,\gamma_1,\dots,\gamma_m)\in J\big\},\\[1ex]
&\beta_\alpha=\delta_0^{-1}H(F)\quad\text{for }\alpha=(F,\gamma_1,\dots,\gamma_m)\in J\,,\\[1ex]
&\rho(r)=\tfrac2\eta r^{-\frac{n+1}{m}}\,.
\end{align*}
\end{proposition}

\begin{proof}
To begin with, observe that for any $0\le j\le m$, any $x_j\in U_j$ and any $0\le \ell\le n+1$ we have that
\begin{equation}\label{MM}
  |F^{(\ell)}_j(x_j)|\le (n+1)M H(F)
\end{equation}
for any $F\in\cF$, where $M$ is given by \eqref{M}. Also, using \eqref{M} and the standard pigeonhole argument (see for example \cite[\S{}I.1~and~\S{}II.1]{Schmidt}) one can easily deduce that for any sufficiently large $t\in\N$ and any $(x_1,\dots,x_m)\in U$ there exists $\vv a\in\Z^{n+1}\setminus\{\vv0\}$ such that
\begin{equation}\label{mink}
  \left\{
  \begin{array}{l}
     |\vv a\cdot\vv f_j(x_j)|<C\, 2^{-\frac{n+1-m}{m}t}\,\qquad(1\le j\le m)\\[0.5ex]
     |\vv a|_\infty \le \delta_02^t
  \end{array}
  \right.
\end{equation}
with
$$
C=4(n+1)M\delta_0^{-\frac{n+1-m}{m}}.
$$

\

Given ${j_0}\in\{1,\dots,m\}$, let $G$ be given by \eqref{G} with
$\ell_{j_0}=n+1-m$ and $\ell_j=0$ for $1\le j\le m$ with $j\not={j_0}$.
Further let $\bmtheta$ be given by \eqref{theta}, where
\begin{align*}
\tilde\theta_{j_0} & =(C\, 2^{-\frac{n+1-m}{m}t},\eta C2^{t},\underbrace{C\,2^t,\dots,C\,2^t}_{n-m\text{ times}})\in\R^{n+2-m}\,,\\[1ex]
\tilde\theta_j & =C\, 2^{-\frac{n+1-m}{m}t}\in\R\qquad (1\le j\le m,~j\not={j_0})
\end{align*}
and $\eta>0$ is to be specified later. Note that $G$ and $\bm\theta$ as defined above depend on ${j_0},t$ and $\eta$. They of course also depend on $C$, $m$, $n$ and the maps $\vv f_j$, but these are fixed throughout the proof. Note that for the above choice of $\bm\theta$ the corresponding parameter $\theta$ defined by \eqref{e:083}${}_{k=n+1}$ is as follows
\begin{equation}\label{theta0}
  \theta=C\eta^{\frac{1}{n+1}}\,.
\end{equation}
Further observe that $\tilde\theta_j$ satisfies Property~{\bf M} for each $j\in\{1,\dots,m\}$ and that, by Proposition~\ref{p:2}, the corresponding parameter $\widehat\Theta$ satisfies
\begin{equation}\label{zz}
    \widehat\Theta\le \max\left\{\frac{2^{-\frac{n+1-m}{m}t}}{C^{n}\eta},\frac{1}{C\,2^t}\right\}\,.
\end{equation}
By Proposition~\ref{p:3}, there is a set $S_{j_0}$ of full measure in $U$ such that for every $\vv x_0\in S_{j_0}$ we have that $\det G(\vv x_0)\not=0$ and there exists a ball $B_{{j_0}}(\vv x_0)\subset U$ centred at $\vv x_0$ and constants $K_{j_0},\alpha_{j_0}>0$ such that for any ball $B\subset B_{j_0}(\vv x_0)$ we have that
\begin{equation}\label{useful2}
\lambda_{\ddd}\Big(B\cap\cA_{{j_0},\eta,t}\Big)\le 2K_{j_0}\,(C\eta^{\frac{1}{n+1}})^{\alpha_{j_0}}\,\lambda_{\ddd}(B)
\end{equation}
for all sufficiently large $t$, where
$$
\cA_{{j_0},\eta,t}=\cA(G,\bmtheta)
$$
with $G$ and $\bmtheta$ depending on ${j_0},\eta,t$ are defined above. Note that if $\vv x=(x_1,\dots,x_m)\in B\setminus\cA_{{j_0},\eta,t}$ then $F\in\cF$ that corresponds to the solution $\vv a$ of \eqref{mink} necessarily satisfies the system
\begin{equation}\label{useful}
  \left\{
  \begin{array}{l}
     |F_{j_0}(x_{j_0})|<C\, 2^{-\frac{n+1-m}{m}t}\,,\\[1ex]
     |F'_{j_0}(x_{j_0})|\ge \eta C\,2^t\,,\\[1ex]
     H(F) \le \delta_02^t\,.
  \end{array}
  \right.
\end{equation}

Define $S=\bigcap_{{j_0}=1}^m S_{j_0}$ and $\cA_{\eta,t}=\bigcup_{{j_0}=1}^m\cA_{{j_0},\eta,t}$. Clearly $S$ is a set of full Lebesgue measure in $U$ as the intersection of sets of full measure. Further, for each $\vv x_0\in S$ we define the ball $\Omega$ as a closed ball centred at $\vv x_0$ and contained in $\bigcap_{{j_0}=1}^m B_{j_0}(\vv x_0)$. Then, since any ball $B$ lying inside $\Omega$ will automatically lie in every $B_{j_0}(\vv x_0)$, by the above argument, for any $\vv x=(x_1,\dots,x_m)\in B\setminus\cA_{\eta,t}$ ~$F\in\cF$ that corresponds to the solution $\vv a$ of \eqref{mink} necessarily satisfies the system \eqref{useful} for every ${j_0}\in\{1,\dots,m\}$. By \eqref{useful2}, there is a fixed choice of $\eta>0$, such that  for every ball $B\subset \Omega$
\begin{equation}\label{useful3}
\lambda_{\ddd}\Big(\tfrac12B\setminus\cA_{\eta,t}\Big)\ge \tfrac12\,\lambda_{\ddd}(\tfrac12B)=2^{-m-1}\lambda_m(B)
\end{equation}
for all sufficiently large $t$, where $\tfrac12B$ is the ball $B$ shrunk by a half.

\medskip

Now fix any $\vv x=(x_1,\dots,x_m)\in\tfrac12B\setminus\cA_{\eta,t}$. Since we are using the supremum norm, $B$ is the product of some intervals $I_j\subset U_j$ of equal lengths, that is $B=I_1\times\cdots \times I_m$. In particular, we have that $x_j\in \tfrac12I_j$. By Taylor's formula, for any $\gamma_j\in I_j$ we have that
$$
F_j(\gamma_j)=F_j(x_j)+F'_j(x_j)(\gamma_j-x_j)+\tfrac12F'_j(\tilde x_j)(\gamma_j-x_j)^2\,,
$$
where $\tilde x_j$ is between $\gamma_j$ and $x_j$. It is readily seen using \eqref{MM} and \eqref{useful}${}_{{j_0}=j}$, that for sufficiently large $t$ we have that
$$
F_j(x_j-\tfrac{2}{\eta}\,2^{-\frac{n+1}{m}t})\qquad\text{and}\qquad
F_j(x_j+\tfrac{2}{\eta}\,2^{-\frac{n+1}{m}t})
$$
have opposite signs. Therefore, by continuity, there exists
$$
\gamma_j\in\big[x_j-\tfrac{2}{\eta}\,2^{-\frac{n+1}{m}t},x_j+\tfrac{2}{\eta}\,2^{-\frac{n+1}{m}t}\big]
$$
such that $F_j(\gamma_j)=0$. Note that, since $x_j\in\tfrac12I_j$, for sufficiently large $t$ we have that $\gamma_j\in I_j$. The collection of all such $\gamma_j$ together with $F$ gives rise to an $\alpha=(F,\gamma_1,\dots,\gamma_m)\in J$ such that
\begin{equation}\label{fact}
\vv x=(x_1,\dots,x_m)\in B(R_\alpha,\rho(2^t))\,,
\end{equation}
where $R_\alpha=(\gamma_1,\dots,\gamma_m)\in B\subset \Omega$
and $\rho$ is as defined in the statement. Hence, by \eqref{fact}, we have that
$$
\tfrac12B\setminus\cA_{\eta,t}\subset\bigcup_{\alpha\in J(t)}B(R_\alpha,\rho(2^t))\,.
$$
By \eqref{useful3}, we immediately conclude \eqref{e:018} with $\kappa_0=2^{-m-1}$ for any ball $B$ in $\Omega$ and all sufficiently large $t$. This verifies the ubiquity hypothesis \eqref{e:018}.

Finally, it remain to verify the technical assumptions that $J$ is countable and $J(t)$ is finite. Note that any $F\in\cF$ will have only a finite number of zeros inside $\Omega$ as $\Omega$ is compact and $F$ is analytic. Therefore, since there are only finitely many $F\in\cF$ with $H(F)\le 2^t$, the set $J(t)$ is finite for any $t\in \N$. Finally, $J$ is countable and a countable union of finite sets $J(t)$.
This completes the proof of the proposition.
\end{proof}

\subsection{Proof of Theorem~\ref{t2+}: the case of $d_1=\dots=d_m=1$}\label{onedim}

Within this subsection we assume that $d_1=\dots=d_m=1$ so that $d=m$, $U_1$,\dots,$U_m$ are intervals in $\R$ and $U=U_1\times\dots\times U_m\subset\R^m$. In particular, $\tilde g=g$. Further, let $\vv f_1,\dots,\vv f_m$, $\cF=\cF(\vv f_1,\dots,\vv f_m)$ be as in Theorem~\ref{t2+}. Without loss of generality, while we prove Theorem~\ref{t2+}, we can assume that \eqref{M} is satisfied. Also, in view of the nature of the conclusion of Theorem~\ref{t2+}, it is sufficient to establish \eqref{vb90} with $U$ replaced by an arbitrarily small neighborhood of almost every point $\vv x_0\in U$. In what follows we shall take $\vv x_0$ such as in Proposition~\ref{ublem} and we let $\Omega$ be as in the proposition. Hence, what we need to prove is
that
\begin{equation}\label{vb90+}
  \cH^g(\cL(\cF,\Psi)\cap\Omega)=\cH^g(\Omega) \qquad \text{if ~~$\sum_{h=1}^\infty h^n g\Big(\frac{\Psi(h)}{h}\Big)=\infty$}\,,
\end{equation}
where $\Psi$ is monotonic. Let $\delta_0=((n+1)M)^{-1}$, $\Phi(h)=\Psi(h)/h$ and $J,\cR,\beta$ and $\rho$ be as in Proposition~\ref{ublem}. Then, $(\cR,\beta)$ is locally ubiquitous in $\Omega$ relative to $\rho$. Since $\Psi$ is decreasing and $g$ is increasing, by Cauchy condensations test, the divergence of the sum in \eqref{vb90+} implies that
$$
\sum_{t=1}^\infty 2^{t(n+1)} g\Big(\frac{\Psi(2^t)}{2^t}\Big)=\infty\,.
$$
In view of the definition of $\rho$ and $\Phi$ above, this further implies the divergence sum condition of \eqref{e:019}. Hence, by Lemma~\ref{l:01}, we have that
\begin{equation}\label{vb91
}
\cH^g\big(\Lambda_\cR(\Phi)\big) = \cH^g(\Omega)\,.
\end{equation}
To conclude \eqref{vb90+} it remains to note that
\begin{equation}\label{vb93}
\Lambda_\cR(\Phi)\subset \cL(\cF,\Psi)\cap\Omega\,.
\end{equation}
Indeed, if $\vv x=(x_1,\dots,x_m)\in \Lambda_\cR(\Phi)$ then there are infinitely many
$(F,\gamma_1,\dots,\gamma_m)$, where $F\in\cF$ and $F_1(\gamma_1)=\dots=F_m(\gamma_m)=0$, such that
$$
\max_{1\le j\le m}|x_j-\gamma_j|<\Phi(\delta_0^{-1}H(F))=\frac{\Psi(\delta_0^{-1}H(F))}{\delta_0^{-1}H(F)}<
\frac{\delta_0\Psi(H(F))}{H(F)}\,.
$$
For each $j$ using the Mean Value Theorem and \eqref{M} we obtain from the above that
$$
|F_j(x_j)|=|F_j(x_j)-F_j(\gamma_j)|=|F'_j(\tilde x_j)(x_j-\gamma_j)|<
(n+1)MH(F)\frac{\delta_0\Psi(H(F))}{H(F)}=\Psi(H(F))\,.
$$
Therefore, \eqref{eq40} is satisfied for infinitely many $F\in\cF$, thus implying that $\vv x\in \cL(\cF,\Psi)\cap\Omega$. This establishes \eqref{vb93} and completes the proof.

\medskip

\subsection{Proof of Theorem~\ref{t2+}: the general case}

The general case will be reduced to that of \S\ref{onedim}. We will therefore need two auxiliary statements.
The first one appears as the Fibering Lemma in \cite{BerInvent}\,:

\begin{lemma}[Fibering Lemma]\label{FL}
Let $f_0,\dots,f_n$ be analytic functions in $k$ real variables defined on an open neighborhood of\/ $\vv0$. Assume that $f_0,\dots,f_n$ are linearly independent over $\R$. Then there is a sufficiently large integer $D_0>1$ such that for every $D\ge D_0$ and every $\vv u=(u_1,u_2,\dots,u_k)\in \R^{k}$ with $u_1\cdots u_k\neq0$ the following functions of one real variable
  $$
  \phi_{\vv u,i}:E_{\vv u}\to\R\quad (0\le i\le n)
  $$
  given by
  $$
  \phi_{\vv u,i}(t)\stackrel{\rm def}{=} f_i(u_1t^{1+D^k},u_2t^{D+D^k},\dots,u_kt^{D^{k-1}+D^k})\,,
  $$
where $E_{\vv u}\subset\R$ is a neighbourhood of\/ $\vv0$, are linearly independent over $\R$.
\end{lemma}

The second auxiliary statement, the so-called `Slicing lemma', is a version of Fubini's theorem for Hausdorff measure and appears as Lemma~4 in \cite{BVSlicing}\,:

\begin{lemma}[Slicing lemma]\label{slicing2}
Let $\, l,k\in\N$ such that $l \leq k $ and $f$ and $\tilde f:r\mapsto
r^{-l} f(r)$ be dimension functions. Let $A\subset\R^k$ be a Borel
set and $V$ be an $(k-l)$-dimensional linear subspace of $\R^k$.
If for a subset $S$ of $V^\perp$ of positive $\cH^{l}$-measure
$$
\cH^{\tilde f}(A\cap(V+b))=\infty\text{ \ \ \ \  for all \ \  }b\in S \,
,
$$
then $\cH^f(A)=\infty$.
\end{lemma}

\bigskip

Now we are fully equipped to complete the proof of Theorem~\ref{t2+}, which will be done by induction on $d$. The case of $d=m$ is considered in \S\ref{onedim}. Hence we assume that some of $d_j$ are strictly bigger than $1$. Without loss of generality we will assume that $d_1>1$. Without loss of generality we will assume that each $U_j$ is centred at $\vv 0$. Next for each $1\le j\le m$ set $u_{j,1}=1$ and consider the following open domains $W_j$ in $\R^{d_j}$ of $(t_j,u_{j,2},\dots,u_{j,d_j})$ given by
\begin{equation}\label{vb95}
u_{j,2}\cdots u_{j,d_j}\neq0
\end{equation}
$$
(t_j^{1+D^{d_j}},u_{j,2}\,t_j^{D+D^{d_j}},\dots,u_{j,d_j}\,t_j^{D^{{d_j}-1}+D^{d_j}})\in U_j^\circ\,,
$$
where $U_j^\circ$ is the interior of $U_j$.
It is readily seen that the map
\begin{equation}\label{vb94}
(t_j,u_{j,2},\dots,u_{d_j}^j)\mapsto(t_j^{1+D^{d_j}},u_{j,2}\,t_j^{D+D^{d_j}},\dots,u_{j,d_j}^j\, t_j^{D^{{d_j}-1}+D^{d_j}})
\end{equation}
is a bijection between the open set $W_j$ and
$$
U_j^*:=\{(x_1,\dots,x_{d_j})\in U_j^\circ:x_2\dots x_{d_j}\neq0\text{ if }x_1\neq0\}.
$$
Also, clearly $U_j^*$ is of the same Lebesgue measure as $U_j$.
Using the change of variables \eqref{vb94}, define the map
\begin{equation}\label{fu}
\tilde{\vv f}_j(t_j,u_{j,2},\dots,u_{j,d_j})=\vv f_j\big(t_j^{1+D^{d_j}},u_{j,2}\,t_j^{D+D^{d_j}},\dots,u_{j,d_j}^j\,t_j^{D^{{d_j}-1}+D^{d_j}}\big)\,,
\end{equation}
where $(t_j,u_{j,2},\dots,u_{j,d_j})\in W_j$. By Lemma~\ref{FL}, for every $\vv u_j=(u_{j,2},\dots,u_{j,d_j})$ subject to \eqref{vb95}
we have that $\tilde{\vv f}_{j,\vv u_j}=\tilde{\vv f}_j(t_j,u_{j,2},\dots,u_{j,d_j})$ as a function of $t_j$ is non-degenerate (see Remark~\ref{rem3} for the definition), where $t_j$ now lies in some interval $\tilde U_j$.
Let $\tilde U=\tilde U_1\times \dots\times \tilde U_m$ and $\tilde\cF$ is defined the same way as $\cF$ but with each $\vv f_j$ replaced by $\tilde{\vv f}_{j,\vv u_j}$. Then, by our induction assumption, we have that
\begin{equation}\label{vb90x}
  \cH^{\tilde g}(\cL(\tilde\cF,\Psi))=\cH^{\tilde g}(\tilde U)\,.
\end{equation}
Then, using either Fubini's theorem (in the case of $\lim_{r\to 0^+}r^{-m}\tilde g(r)<\infty$) or Lemma~\ref{slicing2} (in the case of $\lim_{r\to 0^+}r^{-m}\tilde g(r)=\infty$) completes the proof of Theorem~\ref{t2+} in the general case.

\

\

%\section{Final remarks}
%
%
%
%\noindent\textit{Approximations by conjugate algebraic numbers.}
%Given a function $\Phi:\Rp\to\Rp$, integers $n\ge d\ge 1$, let
%$\A_{n,d}(\Phi)$ denote the set of $(x_1,\dots,x_d)\in \R^d$ such that
%$$
%\max_{1\le\le i\le d}|x_i-\alpha_i|<\Phi(H(\alpha_1))
%$$
%holds for\ infinitely\ many $d$-tuples $(\alpha_1,\dots,\alpha_d)$ of real conjugate algebraic numbers, where $H(\alpha_1)=H(\alpha_2)=\dots=H(\alpha_d)$ denotes the height of the minimal polynomial of $\alpha_1,\dots,\alpha_d$ over $\Z$.
%
%
%\bigskip
%\bigskip
%
%\noindent\textit{Problem 1 (Khintchine type theorem for ABRS-manifolds).} The type of systems of linear forms forms a subclass of manifolds of matrices recently investigated by Aka, Breuillard, Rosenzweig and de Saxc\'e \cite{ABRS2}, to be called \emph{ABRS-manifolds}. These are analytic manifolds of matrices characterised by the property of not being contained in any so-called {\em constraining pencil}. It would be a great interest to extend the findings of this paper to the full class of ABRS-manifolds.
%
%
%\bigskip
%\bigskip
%
%
%
%\noindent\textit{Problem 2 (Badly approximable systems).}
%
%
%

\bigskip
\bigskip

\noindent{\it Acknowledgements.} Part of this work was done during authors' visits to the University of Bielefeld supported by CRC 701 and to the University of York supported by EPSRC grant EP/J018260/1. The third author is also grateful to Maynooth University for their hospitality and providing encouraging environment for working on this project.

\bigskip

%---- plain!, alpha, unsrt, abbrv, siam!, amsalpha

%\bibliographystyle{siam}
%\bibliography{C:/BIBLIOGRAPHY/kleinbock,C:/BIBLIOGRAPHY/beresnevich,C:/BIBLIOGRAPHY/Diophant}

{\small
\def\cprime{$'$}
\def\polhk#1{\setbox0=\hbox{#1}{\ooalign{\hidewidth\lower1.5ex\hbox{`}\hidewidth\crcr\unhbox0}}}

}

\medskip
\medskip
\medskip

{\footnotesize

\begin{minipage}{0.9\textwidth}
\footnotesize V. Beresnevich\\
Department of Mathematics, University of York, Heslington, York, YO10 5DD, England\\
{\it E-mail address}\,:~~ \verb|victor.beresnevich@york.ac.uk|\\
\end{minipage}

\begin{minipage}{0.9\textwidth}
\footnotesize V. Bernik\\
Institute of Mathematics, Surganova 11, Minsk, 220072, Belarus\\
{\it E-mail address}\,:~~ \verb|bernik.vasili@mail.ru|\\
\end{minipage}

\begin{minipage}{0.9\textwidth}
\footnotesize N. Budarina\\
Institute for Applied Mathematics, Khabarovsk Division, Far-Eastern Branch of the Russian Academy of Sciences, Dzerzhinsky st. 54, Khabarovsk, 680000, Russia\\[0.5ex]
{\it E-mail address}\,:~~ \verb|buda77@mail.ru|
\end{minipage}

}

\end{document}